\documentclass[a4paper,reqno,11pt]{amsart}
\usepackage[T1]{fontenc}
\usepackage[margin=1.3in]{geometry}
\usepackage[utf8]{inputenc}	
\usepackage{indentfirst} 
\usepackage{amsmath,amsfonts,amssymb,,mathrsfs}
\usepackage{bbm}
\usepackage[colorlinks=true]{hyperref}
\usepackage{enumerate}
\usepackage{graphicx}
\usepackage{xcolor}
\usepackage{stackengine}
\usepackage{todonotes}
\usepackage[shortlabels]{enumitem}

\usepackage[maxbibnames=99]{biblatex}
\newcommand{\restr}{%
  \,\raisebox{-.127ex}{\reflectbox{\rotatebox[origin=br]{-90}{$\lnot$}}}\,%
}

\allowdisplaybreaks

\theoremstyle{definition}
\newtheorem{definition}{Definition}[section]

\theoremstyle{remark}
\newtheorem{remark}[definition]{Remark}

\theoremstyle{plain}
\newtheorem{theorem}[definition]{Theorem}

\theoremstyle{plain}
\newtheorem{lemma}[definition]{Lemma}
\theoremstyle{remark}

\theoremstyle{plain}
\newtheorem{corollary}[definition]{Corollary}
\theoremstyle{plain}
\newtheorem{proposition}[definition]{Proposition}
\theoremstyle{plain}

\theoremstyle{plain}

\numberwithin{equation}{section}
\allowdisplaybreaks

\def\XXint#1#2#3{{\setbox0=\hbox{$#1{#2#3}{\int}$}
      \vcenter{\hbox{$#2#3$}}\kern-.5\wd0}}

\definecolor{ao}{rgb}{0.0, 0.5, 0.0}

\DeclareMathOperator*{\aplim}{ap-\lim}

\newcommand{\Om}{\Omega}
\newcommand{\R}{\mathbb{R}}

\newcommand{\F}{\mathcal{F}}

\newcommand{\di}{\mathrm{d}}

\newcommand{\de}{\textnormal{d}}
\newcommand{\Ss}{\mathbb{S}}
\newcommand{\N}{\mathbb{N}}
\newcommand{\Q}{\mathbb{Q}}
\newcommand{\Z}{\mathbb{Z}}

\newcommand{\EEE}{\color{black}}

\newcommand{\BBB}{\color{black}}

\def\namedlabel#1#2{\begingroup
    #2%
    \def\@currentlabel{#2}%
    \phantomsection\label{#1}\endgroup
}

\makeatletter
\@namedef{subjclassname@2020}{%
  \textup{2020} Mathematics Subject Classification}
\makeatother

\title[Finite-difference approximation of free/discontinuity problems in $GSBD$]{On De Giorgi's Conjecture of Nonlocal approximations for free-discontinuity problems:\\The symmetric gradient case}

\author[S. Almi]{S. Almi}
\author[E. Davoli]{E. Davoli}
\author[A. Kubin]{A. Kubin}
\author[E. Tasso]{E. Tasso}

\begin{document}

\begin{abstract}
	   
We prove that E. De Giorgi's conjecture for the nonlocal approximation of free-discontinuity problems extends to the case of functionals defined in terms of the symmetric gradient of the admissible field. After introducing a suitable class of continuous finite-difference approximants, we show the compactness of deformations with equibounded energies, as well as their Gamma-convergence. The compactness analysis is a crucial hurdle, which we overcome by generalizing  a Fr\'echet-Kolmogorov approach previously introduced by two of the authors. A second essential difficulty is the identification of the limiting space of admissible deformations, since a control on the directional variations is, a priori, only available in average. A limiting  representation in GSBD is eventually established via a novel characterization of this space.


\EEE

		\medskip
  


	\end{abstract}

 \maketitle

\section{Introduction}
Free-discontinuity problems and their approximation by means of Sobolev formulations, discrete descriptions, or nonlocal functionals are a thriving research area owing to their broad scope of applications, ranging from image reconstruction,  to the modeling of failure phenomena in continuum mechanics. A milestone in this direction was a conjecture formulated by E. De Giorgi and proven by M. Gobbino \cite{Gobbino},  dealing with the approximation via $\Gamma$-convergence of the Mumford-Shah functional by a sequence of nonlocal counterparts. Such analysis has then paved the way for further generalizations in \cite{Cortesani, Gobbino-Mora}. All the aforementioned contributions deal with free-discontinuity problems involving the full distributional gradient of the admissible maps. A question that, to the authors' knowledge, was so far left open, was whether similar nonlocal approximation techniques would also prove successful for the study of free-boundary problems involving more general differential operators. \BBB

In this paper we initiate the study of \emph{continuous finite-difference approximations} of linearized Griffith functionals.  Aside from purely mathematical interest, such inquiries are deeply rooted in the recent research lines which are developing in image processing and data science. Motivated by applications in Magnetic Resonance Imaging (MRI) or Positron Emission Tomography (PET), regularizers involving more general differential operators than the gradient have been studied, e.g., in the settings of regularization graphs \cite{Bredies-Carioni-Holler}, for higher-order Total Variation operators \cite{Brinkmann, Bredies-Holler}, as well as in structural Total Variation approaches \cite{Hintermueller} (see also \cite{data} for a review on data-driven approaches in image regularization).   Concerning the data-science applicative interest, starting from the variational analysis in \cite{Trillos-Slepcev}, a novel research direction has open up for variational studies on point clouds (see, e.g.,~\cite{Braides-Caroccia, Caroccia, Caroccia-Chambolle-Slepcev, MR4069821}), as well as related machine-learning applications \cite{Bungert-Stinson}. The corresponding mathematical study of generalized graph-based differential constraints different from curl $A=0$ is, to the authors' knowledge, a thriving, mostly unexplored, subject (see, e.g., \cite{Liang} and the references therein). \BBB

The case study we tackle here is a symmetrized counterpart of the nonlocal energies considered in~\cite{Gobbino, Gobbino-Mora}. 
In order to describe our findings, we first recall the results \cite{Gobbino}, where the sequence of functionals 
\begin{equation}\label{e:introg1}
    \mathscr{F}_{\varepsilon}(u,\Omega) :=\frac{1}{\varepsilon^{n+1}}\int_{\Omega \times \Omega} \arctan \left ( \frac{( u(  x'   )-u(   x   ))^2}{\varepsilon} \right ) e^{-|\frac{  x' - x  }{\varepsilon}|^2} \de x' \de x, \quad u \in L^1(\Omega)
\end{equation}
is shown to $\Gamma$-converge, as $\varepsilon$ tends to zero, to the Mumford-Shah energy 
\begin{equation}\label{e:introg2}
  \sqrt{\pi}\mathcal{MS}_{1/ \sqrt{\pi}}(u,\Omega) := \int_{\Omega} |\nabla u(x)|^2 \de x + \sqrt{\pi}\mathcal{H}^{n-1}(J_u), \quad u \in \text{GSBV}(\Omega).
\end{equation}
In \eqref{e:introg2}, the acronym ${\rm GSBV}$ stands for the space of functions with generalised special bounded variation (roughly speaking, maps whose truncations are functions of bounded variations with distributional gradients exhibiting null Cantor part), cf.~\cite[Section~4.5]{AFP}. Note that the multiplicative constants in~\eqref{e:introg2} depend only on the choice of the integrand $\arctan(x)$ and on the weight function $e^{-|x|^2}$.

It is thus natural, in a first stage, to ask whether the  functional
\begin{equation}
\label{e:intro1}
    \F_{\varepsilon}(u,\Omega) :=\frac{1}{\varepsilon^{n+1}}\int_{\Omega \times \Omega} \arctan \left ( \frac{( (u(  x'   )-u(   x   )) \cdot (x' - x)   )^2}{\varepsilon^3} \right ) e^{-|\frac{  x' - x  }{\varepsilon}|^2} \de x' \de x
\end{equation}
defined for measurable functions $u \in L^{0} (\Omega; \R^{n})$ provides, analogously, a nonlocal approximation, in the sense of $\Gamma$-convergence, of linearized free discontinuity problems of the form
\begin{equation*}
    \F(u,\Omega) :=\int_{\Omega} \varphi(e(u) ) \, \de x +C\mathcal{H}^{n-1}(J_u), \quad u \in \text{GSBD}(\Omega)
\end{equation*}
for suitable choices of the density $\varphi$ and of the constant $C>0$. In the expression above, $\text{GSBD}(\Omega)$ denotes the space of functions with generalized special bounded deformation (see \cite{DM2013}), and $e(u)$ the absolutely continuous part of the symmetric distributional gradient $\mathcal{E}(u):=\tfrac{Du+Du^{\top}}{2}$.

The structure of \eqref{e:intro1} aligns with the principles of linearized theories in continuum mechanics, where only the symmetric part of the gradient contributes to the deformation energy.  This consideration naturally motivates the assumption that only the component of $u(x') - u(x)$ projected along the direction $x'-x$ should provide energy contributions. By setting $\varepsilon \xi:=x' -x$, when the differences $(u(x + \varepsilon \xi) - u(x)) \cdot \xi$ are relatively small, the functionals in \eqref{e:intro1} behave like pure bulk energies. For large values of $(u(x + \varepsilon \xi) - u(x)) \cdot \xi$, instead, the energies in \eqref{e:intro1} saturate, effectively detecting and penalizing the size of the discontinuities of $u$.

 An important remark concerns the fact that the energies in \eqref{e:introg1} only approximate the Mumford-Shah functional under suitable additional assumptions on the geometry of $\Omega$, such as Lipschitz regularity (see \cite[Remark~7.1]{Gobbino}). Such requirements may not be fully consistent when handling free-boundary problems where discontinuities are already present within the material. Consider, for example, the case of a crack-initiated domain, where the presence of an $(n-1)$-dimensional set $\Gamma \subset \Omega$, representing an initial crack makes the set $\Omega \setminus \Gamma$ not Lipschitz.

 This possible lack of regularity of $\Omega$ calls for a more refined approach, which we address by modifying the energies in \eqref{e:intro1} as follows. First, for every $\xi \in \mathbb{S}^{n-1}$ and every Borel set $E \subset \Omega$ we introduce the functionals $F_{\varepsilon,\xi}(u,E)$, defined as
\begin{equation*}
    F_{\varepsilon,\xi}(u,E):=\frac{1}{\varepsilon} \int_{E \cap (E-\varepsilon\xi)} \arctan  \left ( \frac{((u(x+\varepsilon \xi)-u(x))\cdot \xi )^2}{\varepsilon} \right )  \de x,\quad u \in L^0(\Omega;\mathbb{R}^n).
\end{equation*}
 The specific structure of \eqref{e:intro1} allows one to make use of Fubini's theorem and rewrite it as an average of $F_{\varepsilon,\xi}$ with respect to all possible directions $\xi$. Specifically, we have
\begin{equation*}
    \F_{\varepsilon}(u,E) =\int_{\frac{E-E}{\varepsilon}} F_{\varepsilon,\xi}(u,E) e^{-|\xi|^2} \de \xi,
\end{equation*}
where $E-E:=\{y \in \mathbb{R}^n : y=x'-x, \text{ for some } x,x' \in E\}$. The relevant functionals for our analysis, denoted by $\mathcal{F}^p_\varepsilon(u,\Omega)$, depend on a further degree of freedom, encoded by a parameter $p \in [1,\infty)$, and take the form
\begin{equation}
\label{e:genfun}
    \mathcal{F}^p_\varepsilon(u,\Omega):= \sup_{\mathscr{B}} \sum_{B \in \mathscr{B}} \bigg(\int_{\frac{\Om-\Om}{\varepsilon}} F_{\varepsilon,\xi}(u,B)^p e^{-|\xi|^2} \de \xi\bigg)^{\frac{1}{p}},\quad u \in L^0(\Omega;\mathbb{R}^n).
\end{equation}
In the expression above, $\mathscr{B}$ represents the class of all finite families of pairwise disjoint open balls contained in $\Omega$. This supremum procedure over $\mathscr{B}$ allows the functional $\mathcal{F}^{p}_\varepsilon$ to bypass irregularities within the domain $\Omega$. In fact, by choosing $p=1$, the strict superadditivity of the set function $F_{\varepsilon,\xi}(u,\cdot)$ leads to the strict inequality $\mathcal{F}^1_{\varepsilon} < \mathcal{F}_{\varepsilon}$. 

For $p>1$, the presence of the $L^p$-norm in the definition \eqref{e:genfun} should be viewed as a \emph{regularizing effect}.
\BBB

The major difficulty we face in the analysis of~\eqref{e:genfun} is the lack of a clear compactness strategy. The results of~\cite{AlmiTasso, Chambolle-Crismale, Chambolle-Crismale-2} cannot be applied, as the sequence $\mathcal{F}^{p}_{\varepsilon}$ is defined over the space of measurable maps, so that neither a differential structure nor integrability assumptions are given a priori. Thus, our main result is the following compactness and closure theorem for sequences $u_{\varepsilon}$ with equibounded energy $\mathcal{F}^{p}_{\varepsilon}$, for some $p \in [1, +\infty)$.


\begin{theorem}[Closure and compactness]\label{compattezza}
Let $\Omega \subset \mathbb{R}^n$ be open, let $p\ge1$, and let $\{u_{\varepsilon}\}_{\varepsilon >0} \subset L^0(\Omega;\R^n)$ be  such that
    \begin{equation}\label{energia_limitata}
        \sup_{\varepsilon >0}\F_{\varepsilon}^p(u_{\varepsilon},\Omega) < \infty.
    \end{equation}
    Then, there exists a subsequence $\varepsilon_k \to 0$ as $k \to \infty$ such that the set
    \begin{equation}
    \label{e:setA}
        A := \{x \in \Omega : |u_{\varepsilon_k}(x)|\to \infty \,\, \text{as} \,\, k\to \infty\}
    \end{equation}
    has finite perimeter, and $u_{\varepsilon_k}\to u$ pointwise almost everywhere in $\Omega \setminus A$ for some measurable function $u \colon \Omega \setminus A \to \mathbb{R}^n$. In addition,  for almost every $c \in \R$,  the extension of $u$ to the whole of $\Omega$ defined as $u =c$ on $A$, satisfies $u \in \textnormal{GSBD}(\Omega)$ and 
    \begin{equation}\label{semi378}
    \liminf_{k \to \infty} \mathcal{F}^p_{\varepsilon_k}(u_k,\Omega) \ge \mathcal{F}^p(u,\Omega):=  \int_{\Omega}  \varphi_{p}  (e(u)) \de x + \beta_{p}  \mathcal{H}^{n-1}(J_u \cup \partial^* A),
\end{equation}
where $\varphi_p \colon \mathbb{M}^{n \times n}_{sym} \to [0,\infty)$ is a $2$-homogeneous function and $\beta_p$ is a positive constant. For $p=1$ we have the following explicit form 
\begin{equation*}
   \mathcal{F}^1(u,\Omega)= \frac{\pi^{\frac{n}{2}}}{2}\int_{\Omega} \Big (|e(u)|^2+\frac{1}{2}\textnormal{div}(u)^2\Big)\,\de x +\frac{\pi^{\frac{n+1}{2}}}{2}\mathcal{H}^{n-1}(J_u),
\end{equation*}
whenever $u \in \emph{GSBD}(\Omega)$.
\end{theorem}

For the precise definition of the bulk energy density $\varphi_p$ and of the surface energy density $\beta_p$ we refer to Lemma \ref{limite}.

In proving Theorem~\ref{compattezza} we addressed two main correlated hurdles. The first difficulty concerns the identification of a limit map $u$. In this respect, we may compare our setting to the existing literature concerning nonlocal approximation of Free Discontinuity problems. In particular, the recent works~\cite{Farroni-Scilla-Solombrino, Marziani-Solombrino, Scilla-Solombrino}, inspired by the corresponding results on the Mumford-Shah functional~\cite{Braides-DalMaso, Negri}, provide a nonlocal approximation of a class of Griffith-type functionals in terms of a bulk energy of the form
\begin{displaymath}
\frac{1}{\varepsilon} \int_{\Om} f \big( \varepsilon W(e(u)) * \rho_{\varepsilon}\big)\, \di x  
\end{displaymath}
featuring the convolution of a volume density~$W(e(u))$ for a Sobolev map~$u$ and a potential~$f$ whose behavior is similar to that of $\arctan$ in~\eqref{e:intro1}. The compactness issues in the mentioned articles are overcome due to the presence of an a priori differential structure of the displacement~$u$, relying on compactness in ${\rm GSBD}$~\cite{Chambolle-Crismale, Chambolle-Crismale-2}. We remark that such a direct approach is not applicable in our framework, as $u_{\varepsilon}$ are only measurable fields with no control on their symmetric gradient. A similar issue appeared in a discrete fashion in~\cite{Alicandro-Focardi-Gelli, Crismale-Scilla-Solombrino} on a deterministic lattice, which indeed proposed a discrete finite-difference approximation of Griffith-type of energies in the spirit of~\cite{Bach-Cicalese-Ruf, Chambolle-95, Chambolle-99, Ruf} with a finite range of interaction. Their approach to compactness builds upon the construction of ${\rm GSBD}$-competitors with uniformly bounded energy, which is facilitated by the lack of long-range interactions. This approach allowed again for the application of compactness in ${\rm GSBD}$. The continuous nature of $\mathcal{F}^{p}_{\varepsilon}$ and its infinite horizon make the adaptation of the strategy of~\cite{Crismale-Scilla-Solombrino} rather challenging.


It turned out that in our setting we can work directly on the sequence~$u_{\varepsilon}$, bypassing the modification approach of~\cite{Crismale-Scilla-Solombrino}. Our method relies on Fr\'echet-Kolmogorov, and thus amounts to prove equi-continuity of translations. This technical step is provided in the proof of Theorem~\ref{compattezza} and shares common ideas with~\cite{AlmiTasso}, where the proof of the compactness theorem of \cite{Chambolle-Crismale} was first revisited, avoiding the use of Korn and Korn-Poincar\'e inequalities. Once again, the lack of a symmetric gradient and the structure of~$\mathcal{F}_{\varepsilon}$, which features an extra integration over the directions~$\xi \in \R^{n}$, do not allow us to select a preferred basis to control translations, as it was the case in~\cite{AlmiTasso}. Nevertheless, the slicing properties of the functionals~$\mathcal{F}_{\varepsilon}$ make it possible to control the translations of the maps~$\arctan(u_{\varepsilon}(x)\cdot \xi )$ both with respect to $x$ and~$\xi$. Pointwise convergence out of an exceptional set $A\subseteq \Om$ is thus recovered by Fr\'echet-Kolmogorov Theorem. 

The next step in the proof of Theorem~\ref{compattezza} consists in showing that $u$ belongs to ${\rm GSBD}(\Om)$, yielding the characterization of the domain of the $\Gamma$-limit of $\mathcal{F}^{p}_{\varepsilon}$. Broadly speaking, the non-local nature of the approximating functional $\mathcal{F}^p_\varepsilon$ results in a limiting function space consisting of measurable vector fields that exhibit generalized bounded deformation in a weaker sense. To explain this phenomenon more precisely, recall that $u \in {\rm GBD}(\Omega)$ if and only if $u$ is a measurable vector field and there exists a finite Radon measure $\lambda$ on $\Omega$ such that, for every $\xi \in \mathbb{S}^{n-1}$, the generalized directional variation $\hat{\mu}_u^\xi$ of $u \cdot \xi$, as introduced in \cite[Definition 4.10]{DM2013}, satisfies:
\begin{equation}
\label{e:defgbd}
\hat{\mu}^\xi_u(B) \leq \lambda(B), \quad \text{for every Borel set } B \subset \Omega.
\end{equation}

By a slicing argument,
a similar approach to that proposed by M. Gobbino in~\cite{Gobbino} gives that the limit field $u \in L^{0} (\Om; \R^{n})$ satisfies condition~\eqref{e:defgbd} in an averaged form with respect to $\xi$, namely, 
\begin{equation}
\label{e:defgbde}
\int_{\mathbb{S}^{n-1}} \hat{\mu}^\xi_u(B) \, \di \mathcal{H}^{n-1}(\xi) \leq \lambda(B), \quad \text{for every Borel set } B \subset \Omega. 
\end{equation}
The fact that \eqref{e:defgbd} is replaced in our setting by \eqref{e:defgbde} is a huge obstacle to establish the structural properties of the limiting function space. In particular, the technique presented in \cite{DM2013} for deriving the GBD-structural properties fails, and we need to rely on the more recent integral geometric approach introduced in \cite{AlmiTasso2}. The main difference from the approach in \cite{AlmiTasso2} lies in the role of the flatness of $\mathbb{R}^n$, where rescaling and translating the domain preserve straight lines as geodesics, maintaining invariance in the associated GBD-space. Furthermore, the ${\rm GDB}$-theory developed in~\cite{DalMaso} provides a quite general method to relate the jump set of~$u$ to the behavior of codimension one slices, which we are able to adapt to our setting. We remark that such an approach seems not to be successful in a Riemannian manifold. In Propositions~\ref{p:keyprop}--\ref{p:keyprop1}, we leverage the above features to avoid relying on a Korn-Poincar\'e-type inequality to estimate the dimensionality of the set on which the jump of all one-dimensional slices of~$u$ concentrate.
Ultimately, we conclude in Theorem \ref{t:main-emanuele} that the limit map $u$ belongs to $\text{GSBD}(\Om)$, and thus characterize the domain of the $\Gamma$-limit of~$\mathcal{F}^{p}_{\varepsilon}$. It is worth noting that, although it is not used in the present paper, this method also allows one to address the case of jump sets with infinite length, which may be of interest for cohesive models obtained as the $\Gamma$-limit of finite-difference approximating functionals, as in \cite{Gobbino-Mora}. For the detailed definitions of the relevant quantities and the proof of the structural properties, we refer the reader to Sections \ref{s:main results} and \ref{s:structure}, respectively. 
\medskip

\noindent{\bf An alternative approach.} 

\noindent After the first version of this paper was submitted as a Preprint on
arXiv, a new characterization of GSBD with finite jump was proved in \cite{ChaCri25}, from which it
was deduced that our limiting space is GSBD. This characterization was
based on controlling the measures $\hat{\mu}^\xi_u$ through finitely many direction $\xi$, which turns out to be equivalent to a control on average. In this new
version we propose a totally different perspective, by first deriving
fine properties of the limiting space and then proving it coincides with
GSBD via some integro-geometrical arguments, hence providing an
alternative proof to \cite{ChaCri25}.


\medskip

Eventually, we obtain the following $\Gamma$-convergence result for the sequence $\mathcal{F}^{p}_{\varepsilon}$.

\begin{theorem}[$\Gamma$-convergence]
\label{t:maingamma}
   Let $\Omega \subset \mathbb{R}^n$ be open. For every $p\ge 1$ the family of functionals $\{\mathcal{F}^p_\varepsilon\}_{\varepsilon>0}$ $\Gamma$-converges in $L^0(\Omega;\mathbb{R}^n)$ to the Griffith-type functional $\mathcal{F}^p$ defined as  
   \begin{equation}
\label{e:limitF}
    \mathcal{F}^{p} (u, \Om):= \begin{cases}
        \displaystyle \int_{\Omega}  \varphi_{p} (e(u)) \de x +  \beta_{p}  \mathcal{H}^{n-1}(J_u)& \text{for $u \in {\rm GSBD}(\Om)$,}\\[2mm]
        \displaystyle \infty &\text{otherwise in $L^{0} (\Om; \R^{n})$,}
    \end{cases}
\end{equation}
where $\varphi_{p} $ and $\beta_{p}$  are given as in Theorem~\ref{compattezza}.
\end{theorem}

Combining Theorem~\ref{compattezza} with the $\Gamma$-convergence of Theorem~\ref{t:maingamma}, we further deduce convergence of minimisers of~$\mathcal{F}_{\varepsilon}^p$ to minimisers of~$\mathcal{F}^p$ under suitable Dirichlet boundary conditions (cf.~Theorem~\ref{t:convqmin}).


\medskip

\noindent {\bf Organization of the paper.}
The paper is organized as follows. In Section \ref{s:main results}, we collect a few preliminary definitions and results. Section \ref{s:structure} is devoted to establish the main properties of the spaces of limiting deformations. Sections \ref{s:proof-compactness} and \ref{s:gamma-conv} are devoted to the proofs of our compactness, and $\Gamma$-convergence result, respectively.
\medskip

\noindent {\bf Outlook.}
The analysis provided in this work offers the opportunity to explore several further research lines. On the one hand, the compactness argument of Theorem~\ref{compattezza} and the slicing technique of Theorem~\ref{t:maingamma} may be further tested against discrete finite-difference approximations on point clouds of vectorial free discontinuity problems, in the spirit of~\cite{Braides, Caroccia-Chambolle-Slepcev, Trillos-Slepcev} for the Total Variation, the Mumford-Shah, and the perimeter functionals. Such research line is the subject of a preprint which is currently in preparation.

On the other hand, the use of a slicing argument for the study of the variational limit restricts the choice of the approximating sequence~$\mathcal{F}_{\varepsilon}$, which in turn determines the class of admissible densities~$\varphi_{p}$ in~\eqref{e:limitF}. The study of more general approximations and of integral representation formulas in the spirit of~\cite{Cortesani, Alicandro-Focardi-Gelli, Crismale-Scilla-Solombrino} will be the subject of future investigation. A possible strategy to generalize our result to the case of arbitrary Lam\'e constants could be to rely on the presence of further discrete-divergence terms, along the lines of \cite{Alicandro-Focardi-Gelli, Crismale-Scilla-Solombrino}. The extension of our compactness result and of our identification of the limiting domain for such a model would be possible without any change to the structure of the current proof. The associated Gamma-convergence analysis, on the other hand, deserves the development of a novel representation theory.

\section{Definitions and notation}
\label{s:main results}

Recall the functionals $\mathcal{F}^p_\varepsilon$ and $F_{\varepsilon,\xi}$ defined in the introduction. Given  $A \subset \R$, we further set
\begin{equation*}
    F_{\varepsilon}(u,A):=\frac{1}{\varepsilon} \int_{A} \arctan  \left ( \frac{( u(t+\varepsilon )-u(t))^2}{\varepsilon} \right )  \de t,
\end{equation*}
for every measurable function $u\colon B \subset \mathbb{R} \to \R$ such that $A \subset B \cap (B-\varepsilon)$.
For the sequel it is convenient to introduce the following class of Mumford-Shah and Griffith functionals. For $\gamma>0$,
the former reads as
\begin{equation*}
    \mathcal{MS}_{\gamma}(u,\Omega) :=\begin{cases} 
    \gamma \int_{\Omega} |\nabla u(x)|^2 \de x+  \mathcal{H}^{n-1}(J_u), \quad  u \in \text{GSBV}(\Omega),\\
    +\infty, \qquad  u \in L^1_{loc}(\Omega)\setminus \text{GSBV}(\Omega).
    \end{cases}
\end{equation*}
For $\lambda >0$, we define the Griffith functional by
\begin{equation*}
    \mathcal{G}_{\lambda}(u,\Omega) :=\begin{cases} 
    \lambda \int_{\Omega} | e(u)(x)|^2 \de x+ \mathcal{H}^{n-1}(J_u), \quad  u \in \text{GSBD}(\Omega),\\
    \infty, \qquad  u \in L^0(\Omega;\R^n)\setminus \text{GSBD}(\Omega).
    \end{cases}
\end{equation*}

For every  $\xi \in \R^{n} \setminus \{0\}$, every $x \in \mathbb{R}^n$, and every $E\subset \R^{n}$, we define
\begin{displaymath}
E^{\xi}_{x}:= \{ t \in \R: \, x + t\xi \in E\}\,.
\end{displaymath}
For $u \colon E \to \R^{n}$ we define $\hat{u}^{\xi}_{x} \colon E^{\xi}_{x} \to \R$ as
\[
\hat{u}^{\xi}_{x} (t) := u(x + t\xi) \cdot \xi.
\]
For $v \colon E \to \R$,  we denote by $v^{\xi}_{x}\colon E^{\xi}_{x} \to \R$ the function $v^{\xi}_{x} (t) := v(x + t\xi)$. For $A \subset \R$ open and $v \colon A \to \R$ measurable, we set $J^{1}_{v}:=\{t \in J_{u} \cap A: \,| v^{+}(t) - v^{-}(t)| > 1 \}$.
\\

The proof of the following measurability property is postponed to Section \ref{s:structure}.

\begin{lemma}
\label{l:measfin}
    Let $u \colon \Om\to \R^n$ be $\mathcal{L}^{n}$-measurable. Assume that for $\mathcal{H}^{n-1}$-a.e. $\xi \in \mathbb{S}^{n-1}$, for $\mathcal{H}^{n-1}$-a.e.~$y \in \Pi^{\xi}$, the function $\hat{u}^{\xi}_{y}$ belongs to ${\rm BV}_{loc} (\Om^{\xi}_{y})$. Then, for every open set $U \subset \Om$ the map
\begin{equation}
\label{e:measfin}
(x,\xi) \mapsto | D\hat{u}^{\xi}_{x}| (U^{\xi}_{x} \setminus J^{1}_{\hat{u}^{\xi}_{x}}) + \mathcal{H}^{0} (U^{\xi}_{x} \cap J^{1}_{\hat{u}^{\xi}_{x}})
\end{equation}
is $(\mathcal{L}^n \otimes \mathcal{H}^{n-1})$-measurable on $\Om \times \mathbb{S}^{n-1}$. In particular, for every open set $U \subset \Omega$ we have that the map
\begin{equation*}
    \xi  \mapsto \int_{\Pi^\xi} |D \hat u_y^\xi|( U_y^\xi \setminus J^1_{\hat u_y^\xi})+\mathcal{H}^0( U_y^\xi  \cap J^1_{\hat u_y^\xi}) \de \mathcal{H}^{n-1}(y)
\end{equation*}
is $\mathcal{H}^{n-1}$-measurable on $\mathbb{S}^{n-1}$.
\end{lemma}
Next, we recall a measurability property originally stated in \cite[Lemma~3.6]{DM2013}.

\begin{lemma}
\label{l:measdalmaso}
    Let $v \colon \Om\to \R$ be $\mathcal{L}^{n}$-measurable and let $\xi \in \mathbb{S}^{n-1}$. Assume that for $\mathcal{H}^{n-1}$-a.e.~$y \in \Pi^{\xi}$, the function $v^{\xi}_{y}$ belongs to ${\rm BV}_{loc} (\Om^{\xi}_{y})$. Then, for every Borel set $B \subset \Om$ the map
\begin{equation*}
y \mapsto | Dv^{\xi}_{y}| (B^{\xi}_{y} \setminus J^{1}_{v^{\xi}_{y}}) + \mathcal{H}^{0} (B^{\xi}_{y} \cap J^{1}_{v^{\xi}_{y}})
\end{equation*}
is $\mathcal{H}^{n-1}$-measurable on $\Pi^\xi$.
\end{lemma}

We are now in position to define the space ${\rm GBV}^\mathcal{E}{}(\Omega;\mathbb{R}^n)$.


\begin{definition}\label{d:gbvA}
We say that a $\mathcal{L}^n$-measurable function $u \colon \Omega \to \R^n$ belongs to the space $\textnormal{GBV}^{\mathcal{E}}(\Omega;\R^n)$ if  
for  $\mathcal{H}^{n-1}$-a.e.  $\xi \in \Ss^{n-1}$ and for $\mathcal{H}^{n-1}$-a.e. $y \in \Pi^\xi$ the function $\hat u_y^\xi \in \textnormal{BV}_{loc}(\Omega_y^\xi)$ and
    \begin{equation}
    \label{e:gbve}
         \int_{\mathbb{S}^{n-1}} \bigg( \int_{\Pi^\xi} |D \hat u_y^\xi|( \Omega_y^\xi \setminus J^1_{\hat u_y^\xi})+\mathcal{H}^0( \Omega_y^\xi  \cap J^1_{\hat u_y^\xi}) \de \mathcal{H}^{n-1}(y)  \bigg) \de \mathcal{H}^{n-1}(\xi) <\infty.
    \end{equation} 
    Moreover, we say that $u \in \textnormal{GSBV}^{\mathcal{E}}(\Omega;\mathbb{R}^n)$  if $u \in \textnormal{GBV}^{\mathcal{E}}(\Omega;\mathbb{R}^n)$ and $\hat u_y^\xi \in \textnormal{SBV}_{loc}(\Omega_y^\xi)$ for $\mathcal{H}^{n-1}$-a.e.   $\xi \in \mathbb{S}^{n-1}$ and for $\mathcal{H}^{n-1}$-a.e. $y \in \Pi^\xi$.
\end{definition}

\begin{remark}
Observe that the space $\textnormal{GBV}^{\mathcal{E}}(\Omega;\mathbb{R}^n)$ introduced in Definition~\ref{d:gbvA} does not coincide a priori with $\textnormal{GBD}(\Omega)$ introduced in~\cite[Definition~4.1]{DM2013}. This is because the control in \eqref{e:gbve} is not pointwise in $\xi$ but only in average with respect to the $\mathcal{H}^{n-1} \restr \mathbb{S}^{n-1}$-measure. Nevertheless, we will show in Section~\ref{s:griffith} that, under an additional assumption, actually $\textnormal{GSBV}^{\mathcal{E}}(\Omega;\mathbb{R}^n)=\textnormal{GSBD}(\Omega)$. 
\end{remark}

In view of the measurability property contained in Lemma \ref{l:measdalmaso} we provide the following definition.

\begin{definition}
Let $\Om \subset \R^{n}$ open and let $u \colon \Om \to \R^{n}$ be $\mathcal{L}^{n}$-measurable. Then for every $\xi \in \mathbb{S}^{n-1}$ we define the Borel regular measure $\hat{\mu}^\xi_u$ in $\Omega$ as follows. We define $\hat{\mu}^\xi_u$ on $\Omega$ as the unique Borel regular measure on $\Omega$ that satisfies for every open set $U \subset \Omega$ 
\begin{displaymath}
\hat{\mu}^{\xi}_{u} (U) := \int_{\Pi^{\xi}}| D\hat{u}^{\xi}_{y}| (U^{\xi}_{y} \setminus J^{1}_{\hat{u}^{\xi}_{y} } ) + \mathcal{H}^{0} (U^{\xi}_{y} \cap J^{1}_{\hat{u}^{\xi}_{y} }) \, \di \mathcal{H}^{n-1} (y),
\end{displaymath} 
whenever for $\mathcal{H}^{n-1}$-a.e.~$y \in \Pi^{\xi}$ the function $\hat{u}^{\xi}_{y}$ belongs to ${\rm BV}_{loc}(U^{\xi}_{y})$, while $\hat{\mu}^\xi_u(U)=\infty$ otherwise.   
\end{definition}

\begin{remark}
    Since the hypothesis of De Giorgi-Letta's Theorem \cite{AFP} are fulfilled, the extension of the set function $U \mapsto\hat{\mu}^\xi_u(U)$ to a Borel regular measure on $\Omega$ is uniquely guaranteed by the formula
\begin{equation*}
   \hat{\mu}^\xi_u(B) := \inf \{ \hat{\mu}^\xi_{u}(U) : B \subset U, \ U \subset \Omega \text{ open} \}.
\end{equation*}
\end{remark}

\begin{remark}
    In view of~\eqref{e:gbve}, if $u \in {\rm GBV}^{\mathcal{E}} (\Om)$ we have that $\hat{\mu}^{\xi}_u \in \mathcal{M}^{+}_{b} (\Om)$ for $\mathcal{H}^{n-1}$-a.e.~$\xi \in \mathbb{S}^{n-1}$. Moreover, for such $\xi$, we can make use of the measurability property contained in Lemma \ref{l:measdalmaso} to infer
    \[
    \hat{\mu}^\xi_u(B) = \int_{\Pi^{\xi}}| D\hat{u}^{\xi}_{y}| (B^{\xi}_{y} \setminus J^{1}_{\hat{u}^{\xi}_{y} } ) + \mathcal{H}^{0} (B^{\xi}_{y} \cap J^{1}_{\hat{u}^{\xi}_{y} }) \, \di \mathcal{H}^{n-1} (y),
    \]
    for every Borel set $B \subset \Omega$.
\end{remark}

\begin{definition}
    For $p \ge 1$, $u \in {\rm GBV}^{\mathcal{E}} (\Om)$, and $\xi \in \mathbb{S}^{n-1}$ such that $\hat{\mu}^{\xi}_{u} \in \mathcal{M}^{+}_{b} (\Om)$, we define $\hat{\mu}^{p}_{u}$ on $\Omega$ as the unique Borel regular measure on $\Omega$ that satisfies
    \begin{equation}
    \label{e:mup}
        \hat{\mu}^{p}_{u} (U):= \sup_{\mathscr{B}} \sum_{B \in \mathscr{B}} \bigg(\int_{\mathbb{S}^{n-1}} \hat{\mu}^{\xi}_{u}(B)^{p}\, \di \mathcal{H}^{n-1}(\xi) \bigg)^{\frac{1}{p}} \qquad \text{for $U\subset \Om$ open,}
    \end{equation}
    where $\mathscr{B}$ denotes any finite family of pairwise disjoints open balls contained in~$U$.
\end{definition}
Observe that for $p=1$ the following explicit representation of $\hat{\mu}^1_u$ holds true 
\begin{align}
    \label{e:hat-mu-repr}
        \hat{\mu}^{1}_{u} (U) = \int_{\mathbb{S}^{n-1}} \hat{\mu}^{\xi}_{u} (U)\, \di \mathcal{H}^{n-1}(\xi)\qquad \text{for $U\subset \Om$ open}.
\end{align}

\begin{remark}
    \label{r:mup>mu1}
We observe that, by virtue of Jensen's inequality, it holds true that 
\[
\mu^1_u \leq \mathcal{H}^{n-1}(\mathbb{S}^{n-1})^{\frac{p-1}{p}} \mu^p_u,
\]
in the sense of measures for every $p \ge 1$.
    
\end{remark}

\begin{remark}
    Also in this case, by De Giorgi-Letta's Theorem \cite{AFP} the extension of the set function $U \mapsto \hat{\mu}^p_u(U)$ to a Borel regular measure on $\Omega$ is uniquely guaranteed by  
\[
\begin{split}\
\hat{\mu}^{p}_{u}(B) := \inf \{ \hat{\mu}^{p}_{u}(U) : B \subset U, \ U \subset \Omega \text{ open} \}.
\end{split}
\]
  The above formula immediately gives also the uniqueness of the extension to a Borel regular measure on $\Omega$.
\end{remark}

We conclude this section with a remark concerning the definition of the functionals $\mathcal{F}^p_\varepsilon$ as well as the definition of the measures $\hat{\mu}^p_u$.

\begin{remark}\label{r:finnum}
We observe that in \eqref{e:genfun} and \eqref{e:mup} we can equivalently take the supremum over any family of countable pairwise disjoint open balls and the result would be the same. In addition, all the results as well as their proofs remain unchanged if we replace $\mathscr{B}$ with any class made of finitely many pairwise disjoint convex and open sets.  

We also remark that the functional \eqref{e:genfun} could be replaced by
\[\mathcal{F}'^p_\varepsilon(u,\Omega):=\sup_{ \mathscr{B}} \sum_{B \in \mathscr{B}}\bigg(\int_{\frac{B-B}{\varepsilon}}F_{\varepsilon,\xi}(u,B)^p e^{-|\xi|^{2}}\, d\xi\bigg)^\frac{1}{p }\]
without any substantial change in the proofs that follow. 
In particular, with this choice we would have exactly $\mathcal{F}'^1_\varepsilon(u,B)=\mathcal{F}_\varepsilon(u,B)$ for every open ball $B \subset \R^n$. 
\end{remark}

\section{Structure of the space $\textnormal{GBV}^{\mathcal{E}}$}
\label{s:structure}

The present section addresses the study of the space ${\rm GBV}^{\mathcal{E}}(\Om;\R^n)$ introduced in Definition~\ref{d:gbvA}. In particular, we focus on the structure of functions $u \in {\rm GSBV}^{\mathcal{E}}(\Om,\R^n)$.
As a byproduct, we prove the following identification Theorem.

\begin{theorem}
\label{t:main-emanuele}
  Let $\Om\subset \R^{n}$ open. A measurable function $u \colon \Omega \to \mathbb{R}^n$ belongs to $ {\rm GSBD}(\Omega)$ if and only if $u \in {\rm GSBV}^{\mathcal{E}}(\Omega;\mathbb{R}^n)$.
\end{theorem}

In the next subsection, we discuss some measurability issues which justify formula~\eqref{e:gbve} in Definition~\ref{d:gbvA}.

\subsection{Preliminaries}
We start with a technical proposition for one-dimensional functions. 
We introduce the class $\mathcal{T}$ of all functions $\tau \in C^{1}(\R)$ such that $-\frac12 \leq \tau \leq \frac12$ and $0 \leq \tau'\leq 1$. We recall that, given an open set $I \subset \mathbb{R}$ and given $(\tau_j) \subset C^1(I)$, $\tau \in C^1(I)$, we have $\tau_j \to \tau$ in $C^1_{loc}(I)$ if and only if for every open set $U \Subset I$ it holds $\|\tau_j-\tau\|_{L^\infty(U)} + \|\tau'_j - \tau'\|_{L^\infty(U)} \to 0$ as $j \to \infty$.

\begin{proposition}
\label{p:equality-mu-xi0.1}
Let $I \subset \R$ be open and let $u \in {\rm BV}_{loc}(I)$. Let furthermore $\hat{\mathcal{T}} \subset \mathcal{T}$ be any countable set which is dense with respect to the $C^1_{loc}$-topology. For every $U$ open subset of~$I$, it holds
\begin{equation}
\label{e:equality-mu-xi0.1}
|Du|(U \setminus J^1_u) + \mathcal{H}^0(U \cap J^1_u) = \sup_{k \in \mathbb{N}} \sup \sum_{i=1}^{k} | D(\tau_{i} (u)) |(U_{i})\,,
\end{equation}
where the second supremum is taken over all the families $\tau_{1}, \ldots \tau_{k} \in \hat{\mathcal{T}}$ and all the families of pairwise disjoint open subsets $U_{1}, \ldots, U_{k}$ of~$U$.
\end{proposition} 

\begin{proof}
   Since $u \in  {\rm BV}_{loc}(I)$, for every $U \Subset I$ open and compactly contained in $I$, the condition (a) of \cite[Theorem 3.5]{DM2013} is satisfied. Therefore, we can apply \cite[Theorem 3.8]{DM2013} to the one-dimensional function $u \restr U \in {\rm BV}(U)$ to deduce that \eqref{e:equality-mu-xi0.1} holds true for every open set $U' \subset U$ if we replace $\hat{\mathcal{T}}$ with the entire family $\mathcal{T}$. In order to pass from a compactly contained open set $U$ to the whole of $I$, it is enough to consider a sequence $U_k \Subset I$ with $U_k \nearrow I$ and notice that
   \[
   |Du|(U_k \setminus J^1_u) + \mathcal{H}^0(U_k \cap J^1_u) \nearrow |Du|(I \setminus J^1_u) + \mathcal{H}^0(I \cap J^1_u), \ \ \text{ as $k \to \infty$.}
   \]
   In order to conclude the proof of formula \eqref{e:equality-mu-xi0.1} we need to show that it is enough to consider the supremum on the smaller family $\hat{\mathcal{T}}$. To this purpose we simply notice that for every $U \Subset I$, being $u \in {\rm BV}(U)$, we have $\|u\|_{L^\infty(U)} \leq m$. Hence, by applying the chain rule for ${\rm BV}$-function (see for instance \cite{AFP}), if $(\tau_j) \subset \hat{\mathcal{T}}$ is such that $\tau_j \to \tau$ in $C^1_{loc}(\mathbb{R})$, then 
   \[
   |D(\tau_j(u) - \tau(u))|(U) \leq \|\tau_j -\tau\|_{C^1((-m,m))} \, |Du|(U) \to 0, \ \ \text{ as $j \to \infty$.}
   \]
   With this information at hand we infer that the double supremum in the right hand-side of \eqref{e:equality-mu-xi0.1} does not change when restricted to $\hat{\mathcal{T}}$.
\end{proof}

We are now in position to prove Lemma \ref{l:measfin}.

\begin{proof}[Proof of Lemma \ref{l:measfin}.]
The assumptions of the lemma are equivalent to: for $\mathcal{H}^{n-1}$-a.e. $\xi \in \mathbb{S}^{n-1}$, for $\mathcal{L}^{n}$-a.e.~$x \in \Omega$, the function $\hat{u}^{\xi}_{x}$ belongs to ${\rm BV}_{loc} (\Om^{\xi}_{x})$.

    We claim that for every open set $U \subset \Omega$ and every $\tau \in \mathcal{T}$ the map $(x,\xi) \mapsto | D\tau(\hat{u}^\xi_x) |(U^\xi_x)$ is $(\mathcal{L}^n \otimes \mathcal{H}^{n-1})$-measurable on $\Om \times \mathbb{S}^{n-1}$. To show the claim, consider a countable dense subset $D \subset C^1_c(\mathbb{R})$ in the $C^1_{loc}$-topology. Notice that, for every $f \in L^1_{loc}(\mathbb{R})$ and every $I \subset \mathbb{R}$ open, we have 
    \begin{equation}
    \label{e:supmis}
    | D f |(I) = \sup_{\substack{\varphi \in C^1_c(I) \\ \|\varphi\|_{L^\infty} \leq 1}} \int_{\mathbb{R}} f \, D\varphi \, \de t =  \sup_{\substack{\varphi \in C^1_c(I) \cap D \\ \|\varphi\|_{L^\infty} \leq 1}} \int_{\mathbb{R}} f \, D\varphi \, \de t.
    \end{equation}
    Now consider a sequence of open sets $U'_k \Subset U$ with $U'_k \nearrow U$, and consider a sequence of functions $v_k \in C^1_c(U)$ with $0 \leq v_k \leq 1$ while $v_k=1$ on $U'_k$. Notice that, for every $\varphi \in C^1_c(\mathbb{R})$, the map $(x,\xi) \mapsto \int_{\mathbb{R}} \tau(\hat{u}^\xi_x) \, D((v_k)^\xi_x \varphi) \, \de t$ is $(\mathcal{L}^n \otimes \mathcal{H}^{n-1})$-measurable (for instance by virtue of Fubini's Theorem). By formula \eqref{e:supmis}, we have for every $(x, \xi) \in \Omega \times \mathbb{S}^{n-1}$ that
    \[
    |D\tau(\hat{u}^\xi_x)|((U'_k)^\xi_x) \leq h_k(x,\xi):=  \! \! \! \! \! \!\sup_{\substack{\varphi \in C^1_c(\mathbb{R}) \cap D \\ \|\varphi\|_{L^\infty} \leq 1}} \int_{\mathbb{R}} \tau(\hat{u}^\xi_x) \, D((v_k)^\xi_x \varphi) \, \de t  \leq  |D\tau(\hat{u}^\xi_x)|(U^\xi_x).
    \]
    Notice that, since the countable supremum of measurable functions is a mesurable function, $h_k(x,\xi)$ is $(\mathcal{L}^n \otimes \mathcal{H}^{n-1})$-measurable. By the monotone property of measures, we know that $|D\tau(\hat{u}^\xi_x)|((U'_k)^\xi_x) \nearrow |D\tau(\hat{u}^\xi_x)|(U^\xi_x)$ for every $(x,\xi) \in \Omega \times \mathbb{S}^{n-1}$ as $k \to \infty$, meaning that $h_k(x,\xi) \to |D\tau(\hat{u}^\xi_x)|(U^\xi_x)$ for every $(x,\xi) \in \Omega \times \mathbb{S}^{n-1}$ as $k \to \infty$. Thanks to the fact that pointwise limits of sequences of measurable functions are measurable functions, the claim follows.
    
   By combining the previous claim with Proposition \ref{p:equality-mu-xi0.1}, we deduce that the map in \eqref{e:measfin} is the countable supremum of measurable maps, therefore, it is measurable. The proof is thus concluded.
\end{proof}


We discuss here the measurability of $\xi \mapsto \hat{\mu}^{\xi}_{u} (B)$ for $B$ Borel subset of~$\Om$.

\begin{proposition}
\label{p:equality-mu-xi}
Let $\Om \subset \R^{n}$ open and let $u \colon \Om \to \R^{n}$ be $\mathcal{L}^{n}$-measurable and such that for $\mathcal{H}^{n-1}$-a.e.~$\xi \in \mathbb{S}^{n-1}$ and $\mathcal{H}^{n-1}$-a.e.~$y \in \Pi^{\xi}$ the function $\hat{u}^{\xi}_{y}$ belongs to ${\rm BV}_{loc}(\Om^{\xi}_{y})$. For  every  $\xi \in \mathbb{S}^{n-1}$ and for every $U$ open subset of~$\Om$, it holds
\begin{equation}
\label{e:equality-mu-xi}
\hat{\mu}^{\xi}_{u} (U) = \sup_{k \in \mathbb{N}} \sup \sum_{i=1}^{k} | D_{\xi} (\tau_{i} (u\cdot \xi)) |(U_{i})\,,
\end{equation}
where the second supremum is taken over all the families $\tau_{1}, \ldots \tau_{k} \in \mathcal{T}$ and all the families of pairwise disjoint open subsets $U_{1}, \ldots, U_{k}$ of~$U$.
\end{proposition}

\begin{proof}
Whenever $\xi \in \mathbb{S}^{n-1}$ and $U \subset \Om$ are such that $\hat{\mu}^{\xi}_{u}(U) <\infty$, equality~\eqref{e:equality-mu-xi} holds in view of~\cite[Theorem~3.8]{DM2013}. It remains to consider the case $\hat{\mu}^{\xi}_{u} (U) = \infty$. For simplicity of notation, we denote by $\eta$ the set function given by the right-hand side of~\eqref{e:equality-mu-xi}. By contradiction, assume that $\eta(U)<\infty$. In particular, for every $\tau \in \mathcal{T}$ we have
\begin{equation}\label{qua}
|D_{\xi} \tau(u\cdot\xi)|(U)\leq  \eta(U) <  \infty \,.
\end{equation}
Hence, $D_{\xi} \tau(u\cdot \xi) \in \mathcal{M}_{b} (U)$. By a Carath\'eodory construction using $\eta$ as Gauge function and the family of all open balls contained in $U$ as set of generators, we obtain a Borel regular measure $\lambda$ on $U$, with $\lambda_{\delta}$ as approximating measure for $\delta>0$ (see, e.g., \cite[Section~2.10.1]{Federer}). For every $V \Subset U$, for every $\delta>0$ small enough, and for every covering $\mathcal{G}$ of $V$ made of open balls in $\R^{n}$ with diameter smaller than $\delta$, we may assume that $\bigcup_{A \in \mathcal{G}} A \subset U$. By Besicovitch covering theorem, there exists a dimensional constant $c(n)>0$ and $\mathcal{G}_{1}, \ldots, \mathcal{G}_{c(n)}$ disjoint countable subfamilies of~$\mathcal{G}$ such that
\begin{displaymath}
V \subset \bigcup_{i=1}^{c(n)} \bigcup_{A \in \mathcal{G}_{i}} A\,.
\end{displaymath}
In particular, we have that $\lambda_{\delta} (V) \leq c(n) \eta(U)$, which yields the inequality $\lambda(V) \leq c(n) \eta(U)$ for every $V \Subset U$. Taking the limit $V \nearrow U$ we conclude that $\lambda$ is a positive bounded Radon measure on~$U$.  By using \eqref{qua} it is not difficult to show that actually for every Borel set $B \subset U$ we have that $|D_{\xi} \tau(u\cdot \xi)| (B) \leq \lambda(B)$ whenever $\tau \in \mathcal{T}$. Thus, we are in a position to apply~\cite[Theorem~3.5]{DM2013} to deduce that $\hat{\mu}^{\xi}_{u} (U) \leq \lambda(U) <\infty$, which is a contradiction. Thus, it must be $\eta(U) = \infty$ and equality~\eqref{e:equality-mu-xi} is satisfied.
\end{proof}

\begin{corollary}
\label{c:measurability-mu-xi}
Under the assumptions of Proposition~\ref{p:equality-mu-xi}, for every  open set $U \subset \Om$  the map $\xi \mapsto \hat{\mu}^{\xi}_{u} (U)$ is   lower-semicontinuous  on $\mathbb{S}^{n-1}$. 
\end{corollary}

\begin{proof}
We notice that equality~\eqref{e:equality-mu-xi} holds for every $U \subset \Om$ open, for  every  $\xi \in \mathbb{S}^{n-1}$. In particular, the right-hand side of~\eqref{e:equality-mu-xi} is lower-semicontinuous w.r.t.~$\xi \in \mathbb{S}^{n-1}$ for fixed $U\subset \Om$ open, hence $\xi \mapsto \hat{\mu}^{\xi}_{u} (U)$ is also lower-semicontinuous.
\end{proof}

Thanks to Corollary~\ref{c:measurability-mu-xi}, the integral~\eqref{e:gbve} in Definition~\ref{d:gbvA} is now justified.   We further notice that $\hat{\mu}^{1}_{u}$ can be extended to a finite Radon measure on~$\Om$, that we still denote by $\hat{\mu}^{1}_{u}$.

\begin{remark}
We point out that whenever $u \in {\rm GBV}^{\mathcal{E}} (\Om;\R^n)$, the integral formula~\eqref{e:hat-mu-repr} can be extended also for $K \subset \Om$ compact. Indeed, it is enough to take a sequence $U_{k}$ of open subsets of~$\Om$ such that $K = \bigcap_{k \in \mathbb{N}} U_{k}$. Then, for every $k \in \mathbb{N}$ we have that
  \begin{displaymath}
   \hat{\mu}^1_u \bigg(\bigcap_{j=1}^{k} U_{j} \bigg) = \int_{\mathbb{S}^{n-1}} \hat{\mu}^\xi_{u} \bigg(\bigcap_{j=1}^{k} U_{j} \bigg) \, \de\mathcal{H}^{n-1}(\xi)\,.
    \end{displaymath}
    Passing to the limit as $k\to \infty$, by dominated convergence (recall that $\hat{\mu}^{\xi}_{u} (\Om) \in L^{1}(\mathbb{S}^{n-1})$) we deduce that
      \begin{displaymath}
   \hat{\mu}^1_u (K ) := \int_{\mathbb{S}^{n-1}} \hat{\mu}^\xi_{u} ( K ) \, \de\mathcal{H}^{n-1}(\xi)\,.
    \end{displaymath}
\end{remark}

\subsection{Proof of Theorem~\ref{t:main-emanuele}}
\label{s:griffith}
In this section we prove that if $ u \in \textnormal{GSBV}^{\mathcal{E}}(\Omega;\mathbb{R}^n)$ then $u \in  \textnormal{GSBD}(\Omega)$.
Such implication is a consequence of the following two theorems, concerning the structure properties of functions in ${\rm GSBV}^{\mathcal{E}} (\Om;\mathbb{R}^n)$.

\begin{theorem}
\label{t:slicejs}
    Let $u \in {\rm GBV}^{\mathcal{E}}(\Omega;\mathbb{R}^n)$ and assume that $\hat{\mu}_u^1$ is a finite measure. Then for $\mathcal{H}^{n-1}$-a.e. $\xi \in \mathbb{S}^{n-1}$ we have
    \begin{align}
        \label{e:slicejs}
        (J_u)^\xi_y = J_{\hat{u}^\xi_y}, \quad \text{for $\mathcal{H}^{n-1}$-a.e. } y \in \Pi^\xi
    \end{align}
\end{theorem}

\begin{theorem}\label{t:symmdiff}
 Let $u \in \textnormal{GBV}^{\mathcal{E}}(\Omega;\R^n)$. Then there exists $e(u) \in L^1(\Omega;\mathbb{M}_{\textnormal{sym}}^{n \times n})$ such that 
 \[\aplim_{y \to x} \frac{(u(y)-u(x)-e(u)(x)(y-x))\cdot (y-x)}{|y-x|^2}=0 \]
 holds for a.e. $x \in \Omega$. Moreover, for a.e. $\xi \in \R^n \setminus \{0\}$ and for $\mathcal{H}^{n-1}$-a.e. $y \in \Pi^\xi$ we have
 \begin{equation}\label{e:ugg}
 e(u)_y^\xi \xi \cdot \xi = \nabla \hat u^\xi_y \quad \text{$\mathcal{L}^1$-a.e. on } \Omega_y^\xi.
 \end{equation}
 \end{theorem}

Postponing the proofs of Theorems~\ref{t:slicejs} and~\ref{t:symmdiff} to Subsections~\ref{s:jump} and~\ref{s:symmdiff}, respectively, we conclude this section showing how to exploit them to prove Theorem~\ref{t:main-emanuele}.

\begin{proof}[Proof of Theorem~\ref{t:main-emanuele}]
    The only non trivial implication is $u  \in{\rm GSBV}^{\mathcal{E}}(\Omega;\mathbb{R}^n)$ implies $u \in {\rm GSBD}(\Omega)$. From the very definition of $\text{GSBD}(\Omega)$ we need to show the existence of a finite Borel measure $\lambda$ on $\Omega$ such that for every $\xi \in \mathbb{S}^{n-1}$
    \begin{align}
    \label{e:sliceok}
    &\hat{u}^\xi_y \in \text{SBV}_{loc}(\Omega^\xi_y), \quad \text{for $\mathcal{H}^{n-1}$-a.e. $y \in \Pi^\xi$} \\
    \label{e:sliceok1}
        &\hat{\mu}^\xi_u(B) \leq \lambda(B), \quad \text{for every $B \subset \Omega$ Borel}.
    \end{align}

    By combining Theorem \ref{t:slicejs} together with the countably $(n-1)$-rectifiability of $J_u$ (see \cite{DelNin}), we make use of the Area Formula to infer for $\mathcal{H}^{n-1}$-a.e. $\xi \in \mathbb{S}^{n-1}$ that for every Borel set $B \subset \Omega$
    \begin{equation*}
        \int_{\Pi^\xi}|D^j\hat{u}^\xi_y|(B^\xi_y \setminus J^1_{\hat{u}^\xi_y}) + \mathcal{H}^0(B^\xi_y \cap J^1_{\hat{u}^\xi_y}) \, \de\mathcal{H}^{n-1}(y) = \int_{B \cap J_u} (|[u \cdot \xi]| \wedge 1)\,|\nu_u \cdot \xi| \, \de\mathcal{H}^{n-1},
    \end{equation*}
    where $D^j$ denotes the jump part of the distributional derivative. In particular, there exists a dimensional constant $0 < c_n <1$ such that for every $v \in \mathbb{R}^n$ and $\nu \in \mathbb{S}^{n-1}$
    \[
    c_n(|v| \wedge 1) \leq \int_{\mathbb{S}^{n-1}}(|v \cdot \xi| \wedge 1)\,|\nu \cdot \xi| \, \de\mathcal{H}^{n-1}(\xi) \leq |v| \wedge 1.
    \]
    In addition, by Theorem \ref{t:symmdiff} we have also for $\mathcal{H}^{n-1}$-a.e. $\xi \in \mathbb{S}^{n-1}$ and for every Borel set $B \subset \Omega$
    \begin{align*}
        \int_{\Pi^\xi} |\nabla \hat{u}^\xi_y|(B^\xi_y) \, \de\mathcal{H}^{n-1}(y) &= \int_{\Pi^\xi} \bigg( \int_{B^\xi_y} |e(u)^\xi_y\xi \cdot\xi|\,\de t  \bigg)\de\mathcal{H}^{n-1}(y) \\
        & = \int_B |e(u)\xi\cdot \xi| \, \di x.
    \end{align*}
    Therefore, since $\hat{\mu}^1_u$ is a finite measure (recall also $|e(u)| \in L^1(\Omega)$), the Borel measure $\lambda:= |e(u)| \mathcal{L}^n + (| [u]  | \wedge 1 ) \mathcal{H}^{n-1} \restr J_u$ is finite and for $\mathcal{H}^{n-1}$-a.e. $\xi \in \mathbb{S}^{n-1}$ we have
    \begin{align}\label{disqua}
        \hat{\mu}^\xi_u(B) \leq \lambda(B), \quad \text{for every $B \subset \Omega$ Borel}.
    \end{align}
    In order to get the full conditions \eqref{e:sliceok} and \eqref{e:sliceok1} for every $\xi \in \mathbb{S}^{n-1}$, given $\xi \in \mathbb{S}^{n-1}$, we consider $\xi_k \to \xi$ such that each $\xi_k$ satisfies \eqref{e:sliceok}, \eqref{e:sliceok1} and by virtue of Corollary \ref{c:measurability-mu-xi} we infer that
    \[
    \hat{\mu}^\xi_u(U) \leq \liminf_{k} \hat{\mu}^{\xi_k}_u(U) \leq \lambda(U), \quad \text{for every $U \subset \Omega$ open}.
    \]
    Since $\lambda$ is a finite Borel measure, the above inequality can be further extended to every Borel set $B \subset \Omega$ by outer regularity. This proves \eqref{e:sliceok1}.  In order to verify the validity of \eqref{e:sliceok} we notice that, from \eqref{disqua} and by the definition of $\text{GSBV}^{\mathcal{E}}(\Omega;\mathbb{R}^n)$, we know that for a dense set of $\xi \in \mathbb{S}^{n-1}$ conditions \eqref{e:sliceok} and \eqref{e:sliceok1} hold true.  From \cite[Theorem 3.5]{DM2013}, for every such $\xi$ we have
    \begin{equation}
    \label{e:mthm2}
        |D \tau(u \cdot \xi)|(B) \leq \lambda(B), \quad \text{for every $B \subset \Omega$ Borel},
    \end{equation}
    whenever $\tau \in \mathcal{T}$. By exploiting \eqref{e:mthm2} and the lower semi-continuity of $\xi \mapsto |D \tau(u \cdot \xi)|(U)$ for every open set $U \subset \Omega$, we argue as before and show that actually \eqref{e:mthm2} holds for every $\xi \in \mathbb{S}^{n-1}$ and every Borel set $B \subset \Omega$. By using again \cite[Theorem~3.5]{DM2013}, we infer that for every $\xi \in \mathbb{S}^{n-1}$ property \eqref{e:sliceok1} holds true, and also $\hat{u}^\xi_y \in \text{BV}_{loc}(\Omega^\xi_y)$ for $\mathcal{H}^{n-1}$-a.e. $y \in \Pi^\xi$. Eventually, condition \eqref{e:sliceok1} and the specific form of $\lambda$, together with a simple disintegration argument with respect to the projection $\pi_\xi \colon \mathbb{R}^{n} \to \Pi^\xi$, give exactly the validity of \eqref{e:sliceok}.
\end{proof}

\subsubsection{Proof of Theorem~\ref{t:slicejs}}
\label{s:jump}
In slicing the jump set we will follow the line developed in \cite{AlmiTasso2}.  As already explained in the introduction, the main difference is that, in contrast to the case of a generic Riemannian manifold, the flatness of $\mathbb{R}^n$ allows to avoid the use of a Korn-Poincar\'e type of inequality. We start by introducing a class of relevant measures. Before doing this, we recall a measurability lemma.

\begin{lemma}
Let $u \colon \Omega \to \mathbb{R}^m$ be $\mathcal{L}^n$-measurable. Then, for every Borel set $B \subset \Omega$ we have that
\begin{align}
\label{e:meas1000.1}
    &y \mapsto \sum_{ t \in B^\xi_y } (|[\hat{u}^\xi_y(t)]| \wedge 1)\qquad  \text{ is $\mathcal{H}^{n-1}$-measurable} \\
    \label{e:meas1000.2}
    &\xi \mapsto \int_{\Pi^\xi} \sum_{ t \in B^\xi_y } \big ( |[\hat{u}^\xi_y(t)]| \wedge 1 \big) \, \di \mathcal{H}^{n-1}(y) \qquad\text{ is $\mathcal{H}^{n-1}$-measurable}.
\end{align}
\end{lemma}

\begin{proof}
    See \cite[Lemma 4.5]{AlmiTasso2}.
\end{proof}

Given an $\mathcal{L}^n$-measurable function $u \colon \Omega \to \mathbb{R}^n$, by virtue of \eqref{e:meas1000.1} we consider for every $\xi \in \mathbb{S}^{n-1}$ the (outer) Borel regular measure $\eta_\xi$ of $\mathbb{R}^n$ given by
\begin{align}
    \label{e:defeta1}
    \eta_\xi(B) & := \int_{\Pi^\xi} \sum_{ t \in B^\xi_y} \big( |[\hat{u}^\xi_y(t)]| \wedge 1 \big) \, \di \mathcal{H}^{n-1}(y) \qquad  B \subset \Omega \text{ Borel}\,,
    \\
    \label{e:defeta2}
    \eta_\xi(E) & := \inf \, \{\eta_\xi(B) : \, E \subset B, \ B \subset \Om \text{ Borel}\}\,.
\end{align}
We note that the definition of $\eta_\xi$ in \eqref{e:defeta1} does not depend on the representative of $u$ in its Lebesgue class.



\begin{definition}[The outer measure $\mathscr{I}_{u,1}$]
Let $u \colon \Omega \to \mathbb{R}^n$ be measurable and let $\{\eta_\xi\}_{\xi \in \mathbb{S}^{n-1}}$ be the family of measures in~\eqref{e:defeta1}--\eqref{e:defeta2}. 
By virtue of \eqref{e:meas1000.2},
via the classical Caratheodory's construction, we define the (outer) Borel regular measure $\mathscr{I}_{u,1}$ on $\Omega$ as
\begin{equation*}
    \mathscr{I}_{u,1}(E) := \sup_{\delta>0} \, \inf_{G_\delta} \sum_{B \in G_\delta} \int_{\mathcal{S}^{n-1}} \eta_\xi (B) \de\mathcal{H}^{n-1}(\xi),
 \end{equation*}
whenever $E \subset \Om$ and where $G_\delta$ is the family of all countable Borel covers of $E$ made of sets having diameter less than or equal to~$\delta$.   
\end{definition}

\begin{remark}
    For every $u \in \textnormal{GBV}^{\mathcal{E}}(\Omega;\R^n)$ and for every Borel set $B \subset \Omega$, it holds 
    \[
    \mathscr{I}_{u,1}(B)= \int_{\mathcal{S}^{n-1}} \eta_\xi (B) \de \mathcal{H}^{n-1}(\xi) = \int_{\mathcal{S}^{n-1}} \Big (\int_{\Pi^\xi} \sum_{ t \in B^\xi_y} \big( |[\hat{u}^\xi_y(t)]| \wedge 1 \big) \, \di \mathcal{H}^{n-1}(y) \Big )\de \mathcal{H}^{n-1}(\xi).
    \]
\end{remark}

Given $E\subset\R^n$, we define the codimension one slice of $E$ as 
    \[
    E^\xi_s:=\{y \in \Pi^\xi+s\xi:y  \in E \}.
    \]

\begin{proposition}
\label{newprop}
    Let $\xi \in \mathbb{S}^{n-1}$, $s \in \R$ and let $u\in \textnormal{GBV}^{\mathcal{E}}(\Omega; \R^n)$. Then for $\mathcal{H}^{n-1}$-a.e. $y \in  \Omega_s^\xi$ there exists
    \begin{equation}\label{eq:nnnnn}
    \aplim_{\substack{z \to y,\\ z \in (H^{\pm}_\xi + s\xi)\cap \Omega}} u(z)=:u^{\xi\pm}_s(y),
    \end{equation}
    where $H^{+}_\xi:=\{z \in \R^n : z \cdot \xi >0 \}$ and  $H^-_\xi$ is defined similarly.
\end{proposition}

\begin{proof}
    Since $u \in \textnormal{GBV}^{\mathcal{E}}(\Omega;\R^n)$ we can find a basis $\{\xi_1,\ldots,\xi_n\}$ of $\R^n$ such that $\hat{\mu}_u^{\xi_i}(\Omega)<\infty$ for $i=1,\ldots,n$. By applying  \cite[Theorem 5.1]{DM2013}, 
    we infer that 
    \[
    \aplim_{\substack{z \to y,\\ z \in (H^{\pm}_{\xi_i} + s\xi_i)\cap \Omega}} u(z)\cdot \xi_i=:v_{i}^\pm(y) 
    \]
    exists for every $i=1,\ldots, n$ and for $\mathcal{H}^{n-1}$-a.e. $y \in  \Omega_s^\xi$. Since $\xi_1,\ldots,\xi_n$ form a basis we can readily conclude that the limit in \eqref{eq:nnnnn} exists and, in particular, $u^{\xi\pm}_s(y)\cdot \xi_i=v_{i}^\pm(y)$.
\end{proof}

Thanks to Proposition \ref{newprop}, we are now in a position to identify a precise representative for co-dimension one slices through $(n-1)$-planes.

\begin{definition}
    Given $\xi \in \mathbb{S}^{n-1}$ and $u\colon\Omega \to \R^n$, we define for every $s \in \R$ the function $\bar u^\xi_s\colon \Pi^\xi +s\xi \to \Pi^\xi $ as
    \[
    \bar u^\xi_s (y):=\pi_\xi(u_s^{\xi+}(y)), \quad y \in \Omega^\xi_s,
    \]
    where $u_s^{\xi+}$ is defined in \eqref{eq:nnnnn}.
\end{definition}


\begin{proposition}\label{prop:gbve}
    Let $u \in \textnormal{GBV}^{\mathcal{E}}(\Omega;\R^n)$. Then for $\mathcal{H}^{n-1}$-a.e. $\xi \in \mathbb{S}^{n-1}$ and for $\mathcal{L}^1$-a.e. $s \in \R$ the function $\bar u^\xi_s \in \textnormal{GBV}^{\mathcal{E}}(\Omega_s^\xi;\Pi^\xi).$
\end{proposition}
\begin{proof}
First of all we notice that the Borel regular measure on $\mathbb{S}^{n-1}$ defined as $B \mapsto \int_{\mathbb{S}^{n-1}} \mathcal{H}^{n-2}(B \cap \Pi^\xi) \, d\mathcal{H}^{n-1}(\xi)$ is invariant under the action of the orthogonal group. Hence from \cite[Theorem 2.7.7]{Federer} we infer the existence of a dimensional constant $c(n)$ such that
\begin{equation}\label{eq:sfera}
\mathcal{H}^{n-1}(B)= c(n) \int_{\mathbb{S}^{n-1}} \mathcal{H}^{n-2}(B \cap \Pi^\xi) \, d\mathcal{H}^{n-1}(\xi) \quad \forall B \subset \mathbb{S}^{n-1} \text{ Borel set},
\end{equation}
By using \eqref{e:gbve} and \eqref{eq:sfera} it holds for $\mathcal{H}^{n-1}$-a. $\xi \in \mathbb{S}^{n-1}$
\begin{align*}
\infty &>\int_{\Pi^\xi \cap \mathbb{S}^{n-1}} \hat \mu_u^\eta(\Omega) \de \mathcal{H}^{n-2}(\eta) \\
&=\int_{\Pi^\xi \cap \mathbb{S}^{n-1}} \Big(\int_{\Pi^\eta}(\hat \mu_u)_y^\eta(\Omega_y^\eta) \de \mathcal{H}^{n-1}(y)\Big )\de \mathcal{H}^{n-2}(\eta).
\end{align*}
Since $\mathcal{H}^{n-1}\restr \Pi^\xi = \big [\mathcal{H}^{n-2} \restr (\Pi^\eta \cap (\Pi^\xi +s \xi) \big ]\otimes \mathcal{L}^1$, we have
\begin{align*}
&\int_{\Pi^\xi \cap \mathbb{S}^{n-1}} \Big(\int_{\Pi^\eta}(\hat \mu_u)_y^\eta \de \mathcal{H}^{n-1}(y)\Big )\de \mathcal{H}^{n-2}(\eta)\\
&= \int_{\Pi^\xi \cap \mathbb{S}^{n-1}} \Big(\int_{\R} \Big (\int_{\Pi^\eta \cap (\Pi^\xi+s\xi)}(\hat \mu_u)_y^\eta(\Omega_y^\eta) \de \mathcal{H}^{n-2}(y)\Big )\de \mathcal{L}^1(s) \Big )\de \mathcal{H}^{n-2}(\eta).
\end{align*}
Finally we observe that, for every $\xi,\, \eta \in \mathbb{S}^{n-1}$, $u \cdot \eta = \pi_\xi(u^+)\cdot \eta$ $\mathcal{L}^n$-a.e., thus we can substitute in the above integral $(\hat \mu_u)_y^\eta= (\hat \mu_{\bar u_s^\xi})_y^\eta$, which concludes the proof.
\end{proof}

Next we prove a key technical proposition.

\begin{proposition}
\label{p:keyprop}
Assume that Theorem \ref{t:slicejs} holds true up to dimension $n-1$.
Let $u \colon \Omega \to \mathbb{R}^n$ be measurable. Then, it holds 
\begin{equation}
    \label{e:keyprop1000}
    \mathscr{I}_{u,1} \big ( \{x \in \Omega:\text{for every } \xi \in \mathbb{S}^{n-1}, x \notin J_{\bar u_s^\xi} \text{ with } s \textnormal{ s.t. } x\in \Pi^\xi+s\xi\} \big) = 0\,.
\end{equation}
\end{proposition}
\begin{proof}
Let $E$ be the set in \eqref{e:keyprop1000}, i.e.,
\begin{equation}\label{e:e}
E:=\{x \in \Omega:\text{for every } \xi \in \mathbb{S}^{n-1}, x \notin J_{\bar u_s^\xi} \text{ with } s \textnormal{ s.t. } x\in \Pi^\xi+s\xi\}.
\end{equation}

We show that $E$ is $\hat \mu_u^1$-measurable. We proceed as in \cite[Lemma 4.4]{AlmiTasso2}.
Let $\tau:=\arctan$, we define $s^{\pm} : \Omega \times \mathbb{S}^{n-1} \times \mathbb{S}^{n-1} \to \mathbb{R}^n $ and $i^{\pm} : \Omega \times \mathbb{S}^{n-1} \times \mathbb{S}^{n-1} \to \mathbb{R}^n$ as
\[
s_j^{\pm}(x,\xi, \nu) := \limsup_{r \searrow 0} \int_{(\Pi^\xi +(x\cdot \xi) \xi) \cap (H_{\nu}^{\pm}+x) \cap B_r(x) } \tau((\bar u^{\xi}_{x \cdot \xi})_j(z)) \, \de \mathcal{H}^{n-1}(z),
\]

\[
i_j^{\pm}(x,\xi, \nu) := \liminf_{r \searrow 0} \int_{(\Pi^\xi +(x\cdot \xi) \xi) \cap (H_{\nu}^{\pm}+x) \cap B_r(x) } \tau((\bar u^{\xi}_{x \cdot \xi})_j(z)) \, \de \mathcal{H}^{n-1}(z),
\]
for $j=1,\ldots, n$. We notice that $s^{\pm}$ and $i^{\pm}$ are Borel measurable in $\Omega \times \mathbb{S}^{n-1} \times \mathbb{S}^{n-1}$. Indeed, the integrand functions are Borel measurable in the variables $ (x,\xi, \nu, z)$. Hence, Fubini’s theorem implies that the integral functions are Borel measurable in $\Omega \times \mathbb{S}^{n-1} \times \mathbb{S}^{n-1}$. Finally, both liminf and limsup can be computed by restricting $r \in \mathbb{Q}$ because of the continuity of the integrals with respect to $r$. Then, the set
\[
B = \left\{ (x,\xi, \nu) \in \Omega \times \mathbb{S}^{n-1} \times \mathbb{S}^{n-1} :
s^{+}(x,\xi, \nu) = i^{+}(x,\xi, \nu), \ s^{-}(x,\xi, \nu) = i^{-}(x,\xi, \nu), \right.
\]
\[
\left. s^{+}(x,\xi, \nu) \neq s^{-}(x,\xi, \nu), \quad s^{\pm}_{j}(x,\xi, \nu) \in \left( -\frac{\pi}{2}, \frac{\pi}{2} \right) \text{ for $j =1, \ldots, n$} \right\},
\]
is Borel measurable. Furthermore, $\Om \setminus E = \pi_{1} (B)$, where $\pi_{1} \colon \Om \times \mathbb{S}^{n-1} \times \mathbb{S}^{n-1} \to \Om $ is the projection over the first component. Then, by the measurable projection Theorem (see for instance \cite[Section 2.2.13]{Federer}) the set $\Om \setminus E $ is $\hat \mu_u^1$-measurable, and the same holds for $E$.

From the very definition of $E$, we note that it holds $ (E \cap (\Pi^\xi+s\xi)) \cap J_{\bar u_s^\xi} = \emptyset$ for every $\xi \in \mathbb{S}^{n-1}$ and $s \in \R$. Moreover, we know from Proposition \ref{prop:gbve} that for $\mathcal{H}^{n-1}$-a.e. $\xi \in \mathbb{S}^{n-1}$ and for $\mathcal{L}^1$-a.e $s \in \R$ we also have $\bar u_s^\xi \in \textnormal{GBV}^{\mathcal{E}}(\Omega_s^\xi;\Pi^\xi)$.

For simplicity let us denote $v:=\bar u_s^\xi$. Then, by the assumption that Theorem \ref{t:slicejs} holds true, we have, for $\mathcal{L}^1$-a.e. $s \in \R$,
\begin{equation}
\label{e:induction1}
J_{\hat v_y^\eta}=(J_{v})_y^\eta
\end{equation}
for $\mathcal{H}^{n-2}$-a.e. $\eta \in \mathbb{S}^{n-1}\cap \Pi^\xi$ and  $\mathcal{H}^{n-2}$-a.e. $y \in \Pi^\eta \cap (\Pi^\xi+s\xi)$.
Since the map $(z,s) \mapsto u^\xi_s(z)$ coincides with the map $u(z)$ for $\mathcal{H}^{n-1}$-a.e. $z \in \Omega_s^\xi$ and for $\mathcal{L}^1$-a.e. $s \in \mathbb{R}$, and since $\eta \in \Pi^\xi$, we have as well 
\[
 u(y+t\eta) \cdot \eta = \pi_\xi(u^\xi_s(y+t\eta)) \cdot \eta = v(y+t\eta) \cdot \eta, 
\]
for $\mathcal{L}^1$-a.e. $s \in \mathbb{R}$, for $\mathcal{H}^{n-2}$-a.e. $y \in \Pi^\eta \cap (\Pi^\xi +s\xi)$, and for $\mathcal{L}^1$-a.e. $t \in \Omega^\eta_y$.
In particular, by virtue of Fubini's Theorem (up to measurability issues....), and using also \eqref{e:induction1}, we infer that for $\mathcal{H}^{n-2}$-a.e. $\eta \in \mathbb{S}^{n-1}\cap \Pi^\xi$ it holds
\[
J_{\hat u_y^\eta} \cap E_y^\eta=J_{\hat{v}_y^\eta} \cap E_y^\eta = (J_v)^\eta_y \cap E^\eta_y = (J_v \cap E)^\eta_y= \emptyset \quad \text{for } \mathcal{H}^{n-1}\text{-a.e. } y \in \Pi^\eta ,
\]
where in the last equality we have used the relation  $(E \cap (\Pi^\xi+s\xi)) \cap J_{v} = \emptyset$ valid for every $s \in \mathbb{R}$. Summarizing we have obtained that for $\mathcal{H}^{n-1}$-a.e. $\xi \in \mathbb{S}^{n-1}$ 
\begin{equation}
\label{e:finrel}
    J_{\hat u_y^\eta} \cap E_y^\eta = \emptyset, \ \ \text{ for $\mathcal{H}^{n-2}$-a.e. $\eta \in \mathbb{S}^{n-1}\cap \Pi^\xi$ and $\mathcal{H}^{n-1}$-a.e. $y \in \Pi^\eta$}.
\end{equation}
Eventually, by \eqref{eq:sfera},
we infer from \eqref{e:finrel} that
\[
J_{\hat u_y^\xi} \cap E_y^\xi = \emptyset, \ \ \text{ for $\mathcal{H}^{n-1}$-a.e. $\xi \in \mathbb{S}^{n-1}$ and $\mathcal{H}^{n-1}$-a.e. $y \in \Pi^\xi$}.
\]
From the very definition of the measures $\eta_{\xi}$ (cf.~\eqref{e:defeta1}) we obtain that $\eta_\xi(E)=0$ for $\mathcal{H}^{n-1}$-a.e. $\xi \in \mathbb{S}^{n-1}$. As a direct consequence of the definition of $\mathscr{I}_{u,1}$ we conclude $\mathscr{I}_{u,1}(E)=0$. 
\end{proof}
\BBB


Before proving the dimensional estimate on the set where the measure $\mathscr{I}_{u,1}$ is concentrated, we need the following proposition.

\begin{proposition} 
\label{p:complete}
    Let $A \subset \mathbb{R}^n$ be measurable. Consider the real vector space $L^0(A)$ made of all Lebesgue equivalence classes of measurable functions $v \colon A \to \mathbb{R}$ endowed with the metric $\rm{d}(\cdot,\cdot)$ defined as
    \[
    {\rm d}(v_1,v_2) := \int_A |v_1-v_2| \wedge 1 \, \di x, \ \ v_1,v_2 \in L^0(A),
    \]
    which induces the convergence in measure. Then, any finite dimensional vector subspace $V \subset L^0(A)$ is a complete metric space with respect to the distance ${\rm d}(\cdot,\cdot)$.
\end{proposition}
The proposition above is a direct consequence of Riesz Theorem (see for instance \cite[Theorem 1.21]{Rudin}). To keep the presentation self contained, we include a proof in the Appendix \ref{appenda}.

Recall the definition of $\Theta^{*n-1}$  (see \cite{Federer} for further details)
\[
\Theta^{*n-1}(\mu,x) :=\limsup_{r \to 0} \frac{\mu(B_r(x))}{r^{n-1}},
\]
for every measure $\mu \in \mathcal{M}^+(\Omega)$ and $x \in \Omega$.

\begin{proposition}
\label{p:keyprop1}
    Let $u \in {\rm GBV}^{\mathcal{E}}(\Omega;\mathbb{R}^n)$. Then, for $\hat{\mu}^1_u$-a.e. $x \in \Omega$, the condition $\Theta^{*n-1}(\hat{\mu}^1_u,x)=0$ implies that $x \notin J_{\bar u_{s}^\xi}$ for every $\xi \in \mathbb{S}^{n-1}$, with $s \in \R$ such that $x \in \Pi^\xi+s\xi$. 
\end{proposition}

\begin{proof}
In order to simplify the notation we set for every $x \in \Omega$ and $\xi \in \mathbb{S}^{n-1}$
    \begin{align*}
        O^\xi_x(u)&:=|D \hat{u}^\xi_x|(\Omega^\xi_x \setminus J^1_{\hat{u}^\xi_x}) + \mathcal{H}^0(\Omega^\xi_x \cap J^1_{ \hat{u}^\xi_x})   \\
        O_x(u)&:= \int_{\mathbb{S}^{n-1}}|D \hat{u}^\xi_x|(\Omega^\xi_x \setminus J^1_{\hat{u}^\xi_x}) + \mathcal{H}^0(\Omega^\xi_x \cap J^1_{ \hat{u}^\xi_x}) \, \di \mathcal{H}^{n-1}(\xi).
    \end{align*}
    
\subsubsection*{Step 1.} We claim that 
    \begin{align}
    \label{e:osc1}
    & \hat{u}^\xi_x \in \text{BV}_{loc}(\Omega^\xi_x), \  \ \text{ for $(\mathcal{L}^n \otimes \mathcal{H}^{n-1})$-a.e. $(x,\xi) \in \Omega \times \mathbb{S}^{n-1}$} \\
    \label{e:osc1.5}
    & \ \ \ \ \ \ (x,\xi) \mapsto O^\xi_x(u) \ \ \text{ is $(\mathcal{L}^n \otimes \mathcal{H}^{n-1})$-measurable} \\
    \label{e:osc2}
       & \ \ \ \ \ \ \ \ \ \ \ \ \ \int_{\Omega} O_x(u) \, \de x \leq c_n(\Omega) \hat{\mu}^1_u(\Omega).
    \end{align}
    Indeed, from the definition of $\text{GBV}^{\mathcal{E}}(\Omega;\mathbb{R}^n)$, we know that for $\mathcal{H}^{n-1}$-a.e. $\xi \in \mathbb{S}^{n-1}$ we have $\hat{u}^\xi_y \in \text{BV}_{loc}(\Omega^\xi_y)$ for $\mathcal{H}^{n-1}$-a.e. $y \in \Pi^\xi$. Equivalently, for $\mathcal{H}^{n-1}$-a.e. $\xi \in \mathbb{S}^{n-1}$ we have $\hat{u}^\xi_{\pi_\xi(x)} \in \text{BV}_{loc}(\Omega^\xi_{\pi_\xi(x)})$ for $\mathcal{L}^{n}$-a.e. $x \in \Omega$, where $\pi_\xi \colon \mathbb{R}^n \to \Pi^\xi$ denotes the orthogonal projection. Therefore, by using Fubini's Theorem on the product space $\Omega \times \mathbb{S}^{n-1}$ with measure $\mathcal{L}^n \otimes \mathcal{H}^{n-1}$, and the fact that $\hat{u}^\xi_x= \hat{u}^\xi_{\pi_\xi(x)}$, we immediately infer the validity of \eqref{e:osc1}. To prove \eqref{e:osc2}, in view of the measurability \eqref{e:osc1.5} (see Lemma \ref{l:measfin}), it is enough to make use of Fubini's Theorem and exchange the order of integration in $x$ and $\xi$.

\subsubsection*{Step 2.} Let us assume by contradiction that the measure $\hat{\mu}^1_u$ has null $(n-1)$-density at $x$ while the set $\Sigma \subset \mathbb{S}^{n-1}$ defined as
\[
\Sigma:=\{\xi \in \mathbb{S}^{n-1} : x \in J_{\bar{u}^\xi_s}\}
\]
is not empty. Fix $\xi \in \Sigma$.
By the same arguments used to prove the measurability of the set $E$ defined in \eqref{e:e}, it can be shown that  $\Sigma$ is $\mathcal{H}^{n-1}$-measurable.
\BBB


Let $x \in \Omega$, then $x \in \Pi^\xi+s\xi$ with $s=x\cdot \xi$.
Now define for every $0< r <\text{dist}(x, \partial \Omega)$ the rescaled function $u_{r} \colon B_1(0)  \to \mathbb{R}^n$ as $u_{r}(z) := u^+(x+rz)$, and $\bar u_{r} \colon B_1(0) \cap \Pi^\xi  \to \mathbb{R}^n$ as $\bar u_r(z):=\pi_\xi(u^+(x+rz))$.  

From formula \eqref{e:equality-mu-xi}, we deduce that $u_r \in \text{GBV}^{\mathcal{E}}(B_1(0);\mathbb{R}^n)$ and that $\hat{\mu}^\xi_{u_r}(B_1(0))= r^{1-n} \hat{\mu}^\xi_u(B_r(x))$ for $\mathcal{H}^{n-1}$-a.e. $\xi \in \mathbb{S}^{n-1}$.  
Let us further denote the map $\phi \colon \mathbb{R}^n \to \R^n$ defined as
\[
\phi(z):=
\begin{cases}
    \frac{z}{|z|}, &\text{ if } z \ne 0\\
    0, &\text{ if } z=0.
\end{cases}
\]

From our assumption, we have, for $\xi \in \Sigma$, that there exists $\nu \in \Pi^\xi \cap \mathbb{S}^{n-1}$ such that
\begin{equation}
\label{e:jumpcon}
\bar u_{r} \to f_{\nu,\xi} \ \text{ in measure }  
\text{ as $r \to 0^+$, \ where } \ f_{\nu,\xi}(z):= a_{\nu,\xi} \,\text{sign}(z\cdot \nu) + b_{\nu,\xi}
\end{equation}
for some $a_{\nu,\xi} \in \mathbb{R}^n \setminus \{0\}$ and $b_{\nu,\xi} \in \mathbb{R}^n$. 

\subsubsection*{Step 3.}
We claim that there exist a subsequence of radii $r_k \searrow 0$ and   vectors $\{z_1,\dotsc,z_{n2^n}\} \subset B_1(0) \setminus \{0\}$,  such that $\{z_{(j-1) n + 1},\ldots, z_{jn}\}$ are linearly independent vectors contained in the $j$-th orthant of $\R^n$, $z_i$ are points of approximate continuity for $u_{r_k}$, while \eqref{e:osc1} holds true for $x=z_i$ and $u=u_{r_k}$, and $O_{z_i}(u_{r_k}) \to 0^+$ for every $k=1,2,\dotsc$ as $k \to \infty$. Notice that, once the claim is proved, by the Fundamental Theorem of Calculus together with the concavity of the truncation function, there exists $j=1,\ldots, 2^n$ such that for $\mathcal{H}^{n-1}$-a.e. $z \in B_1(0) \cap \Pi^\xi$ and for every $i=(j-1)n+1,\dotsc,jn$, we have $\xi \cdot z_i >0$ and  
\begin{align}
    \label{e:osc3}
    |(u_{r_k}(z_i) - u_{r_k}(z)) \cdot  \phi (z-z_i)| \wedge 1  &\leq |\hat \mu_{{u}_r}^{ \phi (z-z_i)}|(B_1(0) ), \  k=1,2,\dotsc
\end{align}


We notice that the right-hand side of \eqref{e:osc3} is $\mathcal{H}^{n-1}$-measurable as a function of $z$. 
Indeed, we could have chosen the vectors $\{z_1,\dotsc,z_{n2^n}\}$ so that $\eta \mapsto 
|\hat \mu_{{u}_r}^{ \eta}|(B_1(0) )$ are measurable maps between the $\sigma$-algebra of $\mathcal{H}^{n-1}$-measurable subsets of $\mathbb{S}^{n-1}$ and the Borel $\sigma$-algebra of $\mathbb{R}$ (this follows similarly as in the proof of Lemma \ref{l:measfin}). 
Then, $z \mapsto |\hat \mu_{u_r}^{ \phi (z-z_i)}|(B_1(0) )$ is simply the composition of  $\eta \mapsto |\hat \mu_{u_r}^{ \eta}| (B_1(0) )$
with $z \mapsto \phi(z-z_i)$, which is Borel measurable and whose preimages of $\mathcal{H}^{n-1}$-negligible sets are $\mathcal{H}^{n-1}$-negligible.  

    To prove the claim we first notice that the assumption $\Theta^{n-1}(\hat{\mu}^1_u,x)=0$ implies that $\hat{\mu}^1_{u_r}(B_1(0)) \to 0$ as $r \to 0^+$. By letting $\lambda_r:= \hat{\mu}^1_{u_r}(B_1(0))$ we see from \eqref{e:osc2} applied to $u=u_r$ and $\Omega=B_1(0) \cap \Sigma_j$, where $\Sigma_j$ is the $j$-th orthant, that 
    \begin{equation*}
    \mathcal{L}^n(\{z \in B_1(0) \cap \Sigma_j : O_{z}(u_r) > c_n(\Omega)\sqrt{\lambda_r} \}) \leq \sqrt{\lambda_r}, \ \text{ for every $0 < r < \text{dist}(x,\partial \Omega)$}.
    \end{equation*}
    Therefore, by passing to a subsequence $r_k \searrow 0$ such that $\sum_k \sqrt{\lambda_{r_k}} \leq  \mathcal{L}^n(B_1(0))/2^{n+1}$, we have that the set $A:= \bigcap_k \{z \in B_1(0) \cap \Sigma_j: O_{z}(u_{r_k}) \leq c_n(\Omega)\sqrt{\lambda_{r_k}} \}$ satisfies $\mathcal{L}^n(A) >0$. In particular, we can find $\{z_{(j-1)n+1},\dotsc,z_{jn}\} \subset A \cap \Sigma_j$ in generic position such that $z_i$ is an approximate continuity point for $u_{r_k}$ and \eqref{e:osc1} holds true with $u=u_{r_k}$ and $x=z_i$ for every $k=1,2,\dotsc$ and $i=(j-1)n+1,\dotsc,jn$. The fact that $O_{z_i}(u_{r_k}) \to 0^+$ for every $i=(j-1)n+1,\dotsc,jn$ as $k \to \infty$ is a direct consequence of the fact that $\{z_{(j-1)n+1},\dotsc,z_{jn}\} \subset A$. The claim is thus proved. 
    
    \subsubsection*{Step 4.}  
    Fix $j=1, \ldots, 2^n$ such that $\xi \cdot z_{i}>0$ for $i=(j-1)n,\ldots, jn$. For simplicity assume $j=1$. We observe that for every $i = 1, \ldots, n$
     \begin{equation}
    \label{e:osc7}
        \lim_{r \to 0} \hat{\mu}_{u_{r}} ^{\phi^{+} (z - z_{i})} (B_{1} (0)) = 0 \qquad \text{for $\mathcal{H}^{n-1}$-a.e.~$z \in \Pi^{\xi}\cap  B_{1}(0)$}
    \end{equation}
    since $\Theta^{*n-1} (\hat{\mu}^{1}_{u} , x ) = 0$.

    Since $\{z_{1},\dotsc,z_{n}\}$ are linearly independent vectors, we obtain that
    \begin{equation}
    \label{e:osc6}
        \text{$\{\phi (z-z_1),\dotsc, \phi (z-z_n)\}$ is a basis of $\mathbb{R}^n$ for $\mathcal{H}^{n-1}$-a.e. $z \in \Pi^\xi \cap B_1(0)$}.
    \end{equation}
    Therefore, for $\mathcal{H}^{n-1}$-a.e. $z \in \Pi^\xi \cap (B_1(0) \setminus B_{1/2}(0))$ we find real smooth coefficients $\{c_1(z),\dotsc,c_n(z)\}$ such that $ \phi (z)=\sum_i c_i(z) \phi (z-z_i)$.
    In addition, combining \eqref{e:osc3} and \eqref{e:osc6} we infer for every $i=1,\dotsc,n$ that
    \begin{equation*}
        |(u_{r_k}(z_i)-u_{r_k}(z)) \cdot  \phi (z-z_i)| \wedge 1\to 0, \ \text{pointwise for $\mathcal{H}^{n-1}$-a.e.~$z \in \Pi^\xi\cap \big( B_1(0) \setminus B_{\frac{1}{2}}(0)\big )$ as $k \to \infty$}.
    \end{equation*}
    Therefore, for every sufficiently large $k$ (depending on $z$)
    \begin{align*}
        &|u_{r_k}(z) \cdot  \phi (z) - \sum_{i} c_i(z)u_{r_k}(z_i) \cdot  \phi (z-z_i)| \wedge 1 \\
        &\leq \sum_i |c_i(z)| |(u_{r_k}(z)-u_{r_k}(z_i))\cdot  \phi (z-z_i)| \wedge 1\\
        & \leq \sum_i |c_i(z)| |\hat \mu_{u_r}^{ \phi (z-z_i)}|(B_1(0))
        \end{align*}
        By \eqref{e:osc7}, the above inequality yields that for $\mathcal{H}^{n-1}$-a.e.~$z \in \Pi^{\xi}\cap \big ( B_{1}(0) \setminus B_{\frac{1}{2}}(0)\big )$
        \begin{equation*}
            |u_{r_k}(z) \cdot  \phi (z) - \sum_{i} c_i(z)u_{r_k}(z_i)\cdot  \phi (z-z_i)| \to 0, \ \text{ as $k \to \infty$}.
        \end{equation*}
        Thus, by recalling \eqref{e:jumpcon}, we have for $\mathcal{H}^{n-1}$-a.e.~$z \in \Pi^{\xi} \cap \big (B_{1}(0) \setminus B_{\frac{1}{2}}(0)\big )$ 
        \begin{equation}
         \label{e:osc8}
            |f_{\nu,\xi}(z) - \sum_{i} c_i(z)u_{r_k}(z_i) \cdot  \phi (z-z_i)| \to 0, \ \text{ as $k \to \infty$}.
        \end{equation}
        
        We conclude by showing that \eqref{e:osc8} gives a contradiction. 
        Consider the finite-dimensional real vector subspace $V$ of $L^0\big (\Pi^{\xi} \cap \big (B_{1}(0) \setminus B_{\frac{1}{2}} (0)\big ); \R^{n}\big)$ generated by the elements $\{v_{ij}\colon \Pi^{\xi} \cap \big (B_{1}(0) \setminus B_{\frac{1}{2}} (0) \big)  \to \mathbb{R}^{n} : i,j =1,\dotsc,n\}$, where $v_{ij}(z):= c_i(z) \phi_j (z-z_i)$ ($ \phi_j :=  \phi  \cdot e_j$). Condition \eqref{e:osc8}  combined with Proposition \ref{p:complete}  (recall that pointwise convergence implies convergence in measure) implies that the function $f_{\nu,\xi}$ restricted to $\Pi^{\xi} \cap \big (B_{1}(0) \setminus B_{\frac{1}{2}} (0) \big)$ belongs to $V$.
        Since the generators $v_{ij}$ are all smooth functions, any linear combination of them still belongs to $C^\infty(\Pi^{\xi} \cap \big (B_{1}(0) \setminus B_{\frac{1}{2}} (0)\big) ; \R^{n})$, forcing $a_{\nu,\xi}=0$. This is not possible because we assumed at the beginning $a_{\nu, \xi} \neq 0$ and we reach a contradiction. The proof is thus concluded. 
\EEE
\end{proof}

We are now in position to prove Theorem~\ref{t:slicejs}.

\begin{proof}[Proof of Theorem~\ref{t:slicejs}]
We proceed by induction. Observe that in dimension one the theorem is trivially true. Now, assume that the theorem holds true in dimension $n-1$. Therefore, from Proposition \ref{p:keyprop} we infer that the measure $\mathscr{I}_{u,1}$ concentrates on the complement in $\Omega$ of the set $E$ defined in \eqref{e:e}. Since by Proposition \ref{p:keyprop1} we have that $\Omega \setminus E \subset \{x \in \Omega : \Theta^{*n-1}(\hat{\mu}^1_u,x) >0\}$, then by applying the classical density estimates for Radon measure (cf. \cite{Federer}) we infer that $\Omega \setminus E$ is $\sigma$-finite with respect to  $\mathcal{H}^{n-1}$. We are thus in position to apply Besicovitch-Federer structure theorem \cite[]{Federer} to write $\Omega \setminus  E = R \cup U$ where $R \subset \Omega$ is countably $(n-1)$-rectifiable while the Borel set $U \subset \Omega$ satisfies
\[
\mathcal{H}^{n-1}(\pi_\xi(U)) =0, \ \ \text{ for $\mathcal{H}^{n-1}$-a.e. $\xi \in \mathbb{S}^{n-1}$}.
\]
Eventually, from the very construction of $\mathscr{I}_{u,1}$, the above property immediately implies that $\mathscr{I}_{u,1}(U)=0$ and hence $\mathscr{I}_{u,1}$ concentrates on the rectifiable set $R$. In particular we deduce the fundamental inclusion
\begin{equation*}
J_{\hat{u}^\xi_y} \subset R^\xi_y, \ \ \text{ for $\mathcal{H}^{n-1}$-a.e. $\xi \in \mathbb{S}^{n-1}$ and for $\mathcal{H}^{n-1}$-a.e. $y \in \Pi^\xi$}.
\end{equation*}
We can then conclude by arguing as in the proof of \cite[Theorem 4.12]{AlmiTasso2}.
\end{proof}

\subsubsection{Proof of Theorem~\ref{t:symmdiff}}
\label{s:symmdiff}
The proof follows the line of~\cite[Theorem~9.1]{DM2013}, we report here only the main differences.
 
 We set 
 \begin{equation}\label{e:Xi}
 \Xi:=\{\xi \in \R^n \setminus \{0\}: \hat u_y^\xi \in \text{BV}(\Omega_y^\xi) \text{ for a.e. } y \in \Pi^\xi \}.
 \end{equation}
 Observe that $\mathcal{L}^n(\R^n \setminus \Xi)=0$.
 Without loss of generality, we may assume that $u$ is a Borel function with compact support in $\Omega$ and that $\hat u_y^\xi \in \text{BV}(\Omega_y^\xi)$  for every $\xi \in \Xi$ and for every $y \in \Pi^\xi$. For every $x \in \Omega$ we define
 \begin{equation}\label{e:DM9.3}
 \hat u^\xi(x) :=\limsup_{\rho \to 0^+}\frac{1}{2\rho}\int_{-\rho}^\rho u(x+s\xi)\cdot \xi \, \de s,
 \end{equation}
 \begin{equation*}
 e^\xi(x) :=\limsup_{\rho \to 0^+}\frac{1}{2\rho}\int_{0}^\rho \frac{\hat u^\xi(x+s\xi)-\hat u^\xi(x)}{s} \,\de s.
 \end{equation*}
 We observe that $u^\xi$ and $e^\xi$ are Borel functions and have compact support on $\Omega$. In particular,
 \begin{equation}\label{e:DM9.5}
 e^{\rho \xi}(x)=\rho^2 e^\xi(x) \quad \text{for every } \rho >0 \text{ and every } x \in \Omega.
 \end{equation}
 
 By the Lebesgue Differentation Theorem, for every $y \in \Pi^\xi$ we have 
 \begin{equation*}
 (\hat u^\xi)_y^\xi=\hat u^\xi_y \quad \mathcal{L}^1\text{-a.e. in } \Omega_y^\xi.
 \end{equation*}
 Since $\hat u_y^\xi \in \text{BV}(\Omega_y^\xi)$ and $(\hat u^\xi)_y^\xi$ is a good representative of $\hat u_y^\xi$ by \eqref{e:DM9.3}, we obtain that
 \begin{equation}\label{e:DM9.7}
 \nabla \hat u_y^\xi(t)=\lim_{s \to 0} \frac{(\hat u^\xi)_y^\xi(t+s)-(\hat u^\xi)_y^\xi(t)}{s}=(e^\xi)_y^\xi(t)
 \end{equation}
 for every $y \in \Pi^\xi$ and for $\mathcal{L}^1$-a.e. $t \in \Omega_y^\xi$.

As in \cite[Theorem~9.1]{DM2013}, the following parallelogram identity holds
 \begin{equation}\label{e:DM9.14}
 e^{\xi+\eta}(x)+e^{\xi-\eta}(x)=2e^\xi(x)+2e^\xi(x) \quad \text{for a.e. } x \in \Omega,
 \end{equation}
 for every $\xi, \eta \in \R^n$ such that $\xi, \eta, \xi+\eta, \xi-\eta \in \Xi$.


 Recall \eqref{e:Xi}. Let $\xi_1 \in \Xi$.  By induction, we consider 
 \[\xi_k \in \Xi_k:=\Xi \cap \Bigg(\bigcap_{1\le i \le k-1 } \bigcap_{q \in \Q}\Xi+q\xi_i \Bigg)\] 
 and we remark that $\mathcal{L}^n(\R^n \setminus \Xi_k)=0$. Define $X$ as the vector space over $\mathbb{Q}$ generated by $\{\xi_k\}_{k \in \N}$. Since $\mathcal{L}^n(\R^n \setminus \Xi_k)=0$ for every $k \in \N$, the sequence $\{\xi_k\}_{k \in \N}$ can be chosen to be dense in $\R^n$.  We remark that, since $\Xi$ is closed by multiplication with scalars, then by construction $X \subset \Xi$. Since $X$ is countable and owing to \eqref{e:DM9.14}, there exists a Borel set $N \subset \Omega$ such that $\mathcal{L}^n(N)=0$ and the parallelogram identity
  \begin{equation}\label{e:DM9.15}
 e^{\xi +\eta}(x)+e^{\xi-\eta}(x)=2e^\xi(x)+2e^\xi(x)
 \end{equation}
 holds for every $x \in \Omega \setminus N$ and for every $\xi, \eta \in X$.

 Since $e^\xi(x)$ is also positively homogeneous of degree $2$ by \eqref{e:DM9.5}, we deduce by~\cite[Proposition~11.9]{DalMaso} that for every $x \in \Omega \setminus N$ there exists a symmetric bilinear form $B_x:X\times X \to \R$ such that
 \[e^\xi(x)=B_x(\xi,\xi)\]
 for every $\xi \in X$. This implies that for every $x\in \Omega \setminus N$ there exists a symmetric matrix $e(u)(x)\in \mathbb{M}^{n \times n}_{sym}$ such that
 \begin{equation}\label{e:DM9.16}
 e^\xi(x)=e(u)(x) \xi \cdot \xi
 \end{equation}
 for every $\xi \in X$.
 
 Let us fix $\xi_0 \in \Xi$. We want to prove that \eqref{e:DM9.16} holds for $\xi = \xi_0$ and for a.e. $x \in \Omega$. Let $X_0$ be the vector subspace over $\mathbb{Q}$ generated by $X \cup \{\xi_0\}$. Since $X_0$ is countable, there exists a Borel set $N_0 \subset \Omega$, with $N \subset N_0$ and $\mathcal{L}^n(N_0)=0$, such that \eqref{e:DM9.15} holds for every $x \in \Omega \setminus N_0$ and for every $\xi, \eta \in X_0$. Arguing as before, we prove that for every $x \in \Omega \setminus N_0$ there exists a symmetric matrix $A(x) \in  \mathbb{M}^{n \times n}_{sym}$ such that
 \begin{equation}\label{e:DM9.17}
 e^\xi(x)=A(x)\xi \cdot \xi
 \end{equation}
 for every $\xi \in X_0$. Since $X \subset X_0$ and $N \subset N_0$, equalities \eqref{e:DM9.16} and \eqref{e:DM9.17} hold for every $x \in \Omega \setminus N_0$ and for every $\xi \in X$. This implies that $A(x)=e(u)(x)$ for every $x \in \Omega \setminus N_0$. Since \eqref{e:DM9.17} holds for every $x\in \Omega \setminus N_0$ and for every $\xi\in X_0$, we deduce that the same is true for \eqref{e:DM9.16}. Since $\xi_0 \in X_0$, we conclude that \eqref{e:DM9.16} holds for $\xi =\xi_0$ for every $x \in \Omega \setminus N_0$.
 
 By the arbitrariness of $\xi_0$ we have shown that for a.e. $\xi \in \R^n$ we have
 \[e^\xi(x)=e(u)(x)\xi \cdot \xi \quad \text{a.e. in } \Omega. \]
Lastly, \eqref{e:ugg} follows by combining the above equality and \eqref{e:DM9.7}. 

Finally, the fact that $e(u) \in L^1(\Omega;\mathbb{M}^{n\times n}_{sym})$ follows from
\begin{equation*}
\begin{split}
\int_{\Omega}|e(u)(x)| \,\de x &\le c(n) \int_{\Omega}\int_{\Ss^{n-1}} |e(u)(x)\xi \cdot \xi| \,\de \xi \, \de x\\
&= \int_{\Ss^{n-1}} \int_{\Pi^\xi} \int_{\Omega_y^\xi} |e(u)_y^\xi(t)\xi \cdot \xi)| \,\de t \,\de y \, \de \xi \\
&\le \hat{\mu}^1_u(\Omega) <\infty.
\end{split}
\end{equation*}

\section{Proof of Theorem~\ref{compattezza}}
\label{s:proof-compactness}

We divide this section into two subsection. The first one contains a number of preliminary results that will be exploited in the second subsection, where the proof of Theorem~\ref{compattezza} is carried out.

\subsection{Preliminary results}

We start with a general estimate.

\begin{lemma}
\label{l:fundamentalest}
   For every positive integer $n$ it holds true
    \begin{equation}
    \label{e:fundamentalest}
        \sup_{e \in \mathbb{R}^{n}} \int_{B_4(0)\setminus B_{1/4}(0)}  \frac{|e|}{1+|e \cdot \eta|^2}  \, \de\mathcal{H}^n(\eta) < \infty.
    \end{equation}
\end{lemma}
\begin{proof}
    An application of Coarea Formula with the map $f \colon \mathbb{R}^n \to \mathbb{R}$ defined as $f(\eta):= |\eta|$ allows us to write
    \begin{align*}
    \sup_{e \in \R^n}&\int_{B_4(0)\setminus B_{1/4}(0)}  \frac{|e|}{1+|e \cdot \eta|^2}  \, \de\mathcal{H}^n(\eta) \\
    &=\sup_{e \in \R^n} \int_{1/4}^4 \bigg( \int_{\partial B_\rho(0)} \frac{|e|}{1+|e \cdot \eta|^2} \, \di \mathcal{H}^{n-1}(\eta) \bigg) \,\di \rho \\
    &= \sup_{e \in \R^n}\int_{1/4}^4 \bigg( \int_{\partial B_1(0)} \frac{\rho^{n-1}|e|}{1+\rho^2|e \cdot \eta|^2} \, \di \mathcal{H}^{n-1}(\eta) \bigg)\,  \di \rho\\
    & \leq 4^n \sup_{e \in \R^n}\int_{\mathbb{S}^{n-1}} \frac{|e|}{1+|e \cdot \eta|^2} \, \di \mathcal{H}^{n-1}(\eta)  .
    \end{align*}
Hence, it is enough to prove
\begin{equation}
\label{e:fundamentalest1}
  \sup_{e \in \mathbb{R}^n}  \int_{\mathbb{S}^{n-1}}  \frac{|e|}{1+|e \cdot \eta|^2}  \, \de\mathcal{H}^{n-1}(\eta) <\infty.
\end{equation}
We remark that, since the integrand in \eqref{e:fundamentalest1} is rotation invariant, we can assume $e=\lambda e_1$, where $\lambda \in \R$ and $e_1$ is the first element of the canonical basis.
    Moreover, we observe that by setting $\eta_1 := \eta \cdot e_1$ it is enough to show 
    \begin{equation*}
        \sup_{\lambda >0}\int_{\Ss_\delta^{n-1}}  \frac{\lambda}{1+\lambda^2\eta_1^2}  \,  \de\mathcal{H}^{n-1} (\eta) < \infty
    \end{equation*}
    for some $\delta \in (0,1)$,
    where $\Ss_\delta^{n-1}=\Ss^{n-1}\cap \{0\leq \eta_1\leq \delta\}.$ Indeed, the integral      
    \begin{equation*}
        \int_{\Ss^{n-1} \cap \{\eta_1>\delta\}}  \frac{\lambda}{1+\lambda^2\eta_1^2}  \,  \de\mathcal{H}^{n-1}  (\eta) 
    \end{equation*}
    is uniformly bounded by a constant only depending on $n$ and $\delta$.
    
    By applying the Coarea Formula with the map $f \colon \mathbb{S}^{n-1} \to \mathbb{R}$ defined as $f(\eta):= \eta_1$ we have
    \begin{align*}
        \int_{\Ss_\delta^{n-1}} &  \frac{\lambda}{1+\lambda^2\eta_1^2} \frac{|\text{J}_{\mathbb{S}^{n-1}}f(\eta)|}{|\text{J}_{\mathbb{S}^{n-1}}f(\eta)|}   \,  \de\mathcal{H}^{n-1}  (\eta) \\
        &\leq \sup_{\eta \in \Ss_\delta^{n-1}}  \frac{1}{|\text{J}_{\mathbb{S}^{n-1}}f(\eta)|} \int_0^1 \frac{\lambda}{1+\lambda^2t^2}\mathcal{H}^{n-2}(\Ss_\delta^{n-1} \cap \{\eta_1=t\}) \, \de t\\
        &\leq  C(n,\delta)  \int_0^1 \frac{\lambda}{1+\lambda^2t^2} \, \de t=C(n,\delta)\arctan(\lambda) \le C(n,\delta)\frac{\pi}{2}.
    \end{align*}
     This concludes the proof of~\eqref{e:fundamentalest}.
\end{proof}

We now recall some useful results shown in~\cite{Gobbino} for the study of nonlocal approximations of the Mumford-Shah functional. For the reader's convenience, we report the proof of Lemma~\ref{Gob5.3}, highlighting the ambient space $\Omega$ in the functional. The proofs of Lemmas~\ref{Lemma3.2}-\ref{lemmaG5.1} coincide with those of~\cite{Gobbino}.

\begin{lemma}[{\cite[Lemma 3.2]{Gobbino}}]\label{Lemma3.2}
    Let $I=[a,b]$ be an interval, let $\{u_{\varepsilon}\}_{\varepsilon>0} \subset L^1_{loc}(\R)$, and let $u \in L^1_{loc}(\R)$. Let us assume the following:
    \begin{itemize}
        \item[(i)] $u_{\varepsilon} \to u$ in $L^1_{loc}(\R);$
        \item[(ii)] $a$ and $b$ are Lebesgue points of $u$.
    \end{itemize}
    Then
    \begin{equation*}
        \liminf_{\varepsilon \to 0} F_{\varepsilon}(u_{\varepsilon},I) \ge \min \left \{ \frac{\pi}{2}, \frac{(u(b)-u(  a  ))^2}{b-a}\right \}.
    \end{equation*}
\end{lemma}

\begin{lemma}[{\cite[Lemma 3.3]{Gobbino}}]\label{Lemma3.3}
    Let $u \in L^{\infty}(\R).$ Then there exists $a \in \R$ such that
    \begin{itemize}
        \item[(i)] $a+q$ is a Lebesgue point of $u$ for every $q \in \Q$;
        \item[(ii)] every sequence $\{u_k\}_{k \in \N} \subset L^\infty(\R)$ that satisfies the following conditions:
        \begin{itemize}
            \item $u_k(a+\frac{z}{k})=u(a+\frac{z}{k})$ for all $z \in \Z$;
            \item if $x \in [a+\frac{z}{k},a+\frac{z+1}{k}],$ then $u_k(x)$ belongs to the interval with endpoints $u(a+\frac{z}{k})$ and $u(a+\frac{z+1}{k})$;
        \end{itemize}
        has a subsequence converging to $u$ in $L^1_{loc}(\R)$.
    \end{itemize}
\end{lemma}

\begin{lemma}[{\cite[Lemma~5.1]{Gobbino}}]\label{lemmaG5.1}
Let $\Omega \subset \R^n$ be an open set.  For  every $u\in L^0(\Omega, \R^m)$, every $E \Subset \Omega$, every $\delta >0$, and every $\xi \in \R^n$ such that $E + \delta \xi \subset \Om$, we have that 
    \begin{equation*}
        \int_E |\arctan(u(x+\delta \xi) \cdot \xi)-\arctan(u(x)\cdot \xi)| \de x\le C_E \delta (1+F_{\delta,\xi}(u,  E )),
    \end{equation*}
    for some $C_E>0$ only depending on $E$.
\end{lemma}

\begin{lemma}[{\cite[Lemma 5.3]{Gobbino}}]\label{Gob5.3}
Let $\Omega \subset \R^n$ be an open set, $E \Subset \Omega$, $R \subset \R^{n}$, and $\eta >0$ be such that $R\subset \frac{\Omega - \Omega}{\eta}$ and
\[
{\rm dist} ( E, \partial \Om) \geq \eta \, \sup_{\xi \in R} |\xi|\,.
\]
Then, for every $u \in L^{0}(\Omega;\R^n)$, $\varepsilon>0$, and $m  \in \N$ such that $m \varepsilon <\eta$  we have that 
    \begin{equation}
    \label{e:gob5.3}
    \int_{R} F_{ m \varepsilon , \xi}(u,E ) \, \di \xi \le \int_{R} F_{\varepsilon, \xi}(u, \Omega) \, \di \xi \,.
    \end{equation}
\end{lemma}

\begin{proof}
Given $\varepsilon >0$, we prove the statement by induction over $m$. If $m=1$ there is nothing to prove. Let us assume that~\eqref{e:gob5.3} holds for $ m $ and prove it for $ m+1 $,  assuming that $ (m + 1) \varepsilon  < \eta $. In particular, this implies that $E \subset \Om - (m+1) \varepsilon \xi$ for every $\xi \in R$. 
For $A, B \in \R$ it holds
\begin{displaymath}
\arctan \bigg( \frac{(A+ B)^{2}}{ m+1 } \bigg) \leq \arctan (A^{2}) + \arctan \bigg( \frac{B^{2}}{ m } \bigg) \,.
\end{displaymath}
Applying such inequality to 
\begin{align*}
A = \frac{ (u(x + ( m+1 ) \varepsilon \xi) - u(x +  m \varepsilon \xi)) \cdot \xi}{ \sqrt{\varepsilon}} \qquad B = \frac{ (u(x +  m  \varepsilon \xi) - u(x)) \cdot \xi}{\sqrt{\varepsilon}}\,,
\end{align*}
with $x \in E$ and $\xi \in R$, we get that
\begin{align*}
 & \frac{1}{( m+1 )\varepsilon  } \arctan \bigg( \frac{ \big( (u(x +  (m+1) \varepsilon \xi) - u(x) ) \cdot \xi \big) ^{2}}{(m +1)  \varepsilon} \bigg) 
 \\
 &
 \quad \leq \frac{1}{( m +1)  \varepsilon} \arctan \bigg( \frac{ \big( ( u ( x +  (m+1)  \varepsilon \xi) - u( x +  m \varepsilon \xi) ) \cdot \xi \big)^{2}}{\varepsilon} \bigg) 
 \\
 &
 \qquad +  \frac{m}{(m+1)}\frac{1}{ m  \varepsilon} \arctan \bigg( \frac{ \big( (u(x +  m \varepsilon \xi) - u(x)) \cdot \xi \big)^{2}}{ m  \varepsilon} \bigg) 
\end{align*}
Integrating over $x \in E$ and $\xi \in R$ and performing a change of variable in the first integral on the right-hand side we obtain
\begin{align*}
   \int_{R}  F_{( m+1  )\varepsilon, \xi} (u, E) \, \di \xi & \leq  \frac{1}{( m+1  ) \varepsilon} \int_{R} \int_{E +  m \varepsilon \xi} \arctan \bigg( \frac{ \big( ( u ( x +  \varepsilon \xi) - u( x ) ) \cdot \xi \big)^{2}}{\varepsilon} \bigg) \, \di x \, \di \xi
   \\
   &
   \qquad +   \frac{m}{(m+1)} \int_{R} F_{ m \varepsilon, \xi} (u, E) \, \di \xi
   \\
   &
   \leq \frac{1}{( m  +1)} \int_{R} F_{\varepsilon, \xi} (u, \Omega) \, \di \xi +  \frac{m}{(m+1)}  \int_{R} F_{ m \varepsilon, \xi} (u, E) \, \di \xi\,,
\end{align*}
where, in the last inequality, we have used that $E + m \varepsilon \xi \subset \Om \cap \Om - \varepsilon \xi$. Hence, we conclude for~\eqref{e:gob5.3} by the induction hypothesis.
\end{proof}

We conclude this section with the following lemma, which characterizes the  density~$\varphi_{p}$  and the constant~$\beta_p$ appearing in~\eqref{semi378} and in the definition of the functional~$\mathcal{F}^{p}$ in~\eqref{e:limitF}.

\begin{lemma}
\label{limite}
Let $p \ge 1$, let $\Omega \subset \R^n$ be an open set, and let $\varphi_{p} \colon \mathbb{M}^{n \times n}_{sym} \to [0,\infty)$ be the map with quadratic growth defined as 
\begin{align*}
     \varphi_{p} (A)&:= \bigg(\int_{\R^n} \left ( |A \xi \cdot \xi|^2   \right )^{p}|\xi|^p e^{-|\xi|^2} \de \xi \bigg)^{\frac{1}{p}} \qquad  \text{for $A \in \mathbb{M}^{n \times n}_{sym}$ and $p <\infty$}.
\end{align*}  
Then, there exists a positive constant $\beta_{p} >0$ such that for every $u \in  \textnormal{GSBD}(\Omega;\R^n)  $ it holds
    \begin{align*}
         \sup_{\mathscr{B}} \sum_{B \in \mathscr{B}}\frac{\pi}{2}&\left (\int_{\R^n} \left (\int_{\Pi^\xi} \mathcal{MS}_{\frac{2}{\pi}}(\hat u_y^\xi,B_y^\xi) \, \de \mathcal{H}^{n-1}(y) \right )^{p}|\xi|^p e^{-|\xi|^2} \de \xi\right )^{\frac{1}{p}}\\
         &=\!\int_{\Omega}  \varphi_{p}  (e(u))\,  \de x +  \beta_p  \mathcal{H}^{n-1}(J_{u}),
    \end{align*}
 where $\mathscr{B}$ is the set of all possible finite families of pairwise disjoint balls in $\Omega$.
\end{lemma}

\begin{proof}
We recall that 
\begin{equation*}
          \frac{\pi}{2}\mathcal{MS}_{\frac{2}{\pi}}(\hat u_y^\xi,B_y^\xi)=  \int_{B_y^\xi}|\nabla \hat u_y^\xi|^2 \de  \tau + \frac{\pi}{2}\mathcal{H}^{0}(J_{\hat u_y^\xi} \cap B_y^\xi) .
\end{equation*}
Integrating the expression above in $y \in \Pi^\xi$, by Theorem~\ref{t:symmdiff} we obtain for the first term
\begin{equation*}
\begin{split}
 \int_{\Pi^\xi}  \int_{B_y^\xi} |\xi| | \nabla \hat{u}^{\xi}_{y}|^{2}\, \di \tau \, \di \mathcal{H}^{n-1} (y) &
 =  \int_{\Pi^\xi}  \int_{B_y^\xi} |\xi||\nabla u(y+\tau \xi)\cdot \xi|^2 \, \de \tau\,  \de \mathcal{H}^{n-1}(y) 
 \\
 &
 =\int_{\Pi^\xi} \int_{B_y^\xi} |\xi||e(u)(y+\tau \xi )\xi \cdot \xi|^2 \, \de \tau \,  \de \mathcal{H}^{n-1}(y)\\
&
=\int_{B}  |\xi| |e(u)(x )\xi \cdot \xi|^2  \, \de x.
\end{split}
\end{equation*}
On the other hand, by the area formula and Theorem \ref{t:slicejs}, for the second term we get
\begin{equation*}
    \begin{split}
    \int_{\Pi^\xi} |\xi|\mathcal{H}^0(J_{\hat u^\xi_y} \cap B)\,  \de \mathcal{H}^{n-1}(y)=  \int_{J_{u} \cap B} |\xi| |\nu_u(x)\cdot \xi|\,  \de \mathcal{H}^{n-1}(x).
    \end{split}
\end{equation*}

We define for every ball $B \subset \Omega$ the function
\[\zeta(u,B):=\left (\int_{\R^n} \left (\int_{B} |e(u)(x) \xi \cdot \xi|^2 \de x +\frac{\pi}{2}\int_{J_u\cap B} |\nu_u(x) \cdot \xi|\,  \de \mathcal{H}^{n-1}(x) \right )^{p}|\xi|^p e^{-|\xi|^2} \, \de \xi\right )^{\frac{1}{p}},\]
and for every open $A \subset \Omega$ we set
\[\mu(u,A):=\sup_{\mathscr{B}_A} \sum_{B \in \mathscr{B}_A} \zeta(u,B),\]
where $\mathscr{B}_A$ is any finite family of disjoint balls contained in $A$. 
Then, $\mu$ can be extended to a Borel measure (which we still denote by $\mu$ with a  slight  abuse of notation). We also observe that there exists a positive constant $C$ only depending on $n$ and $p$ such that for every open set $A \subset \Omega$ we have
\begin{equation}\label{e:dis2}
     \mu(u,A) \le C \mathcal{G}_{\frac{2}{\pi}}(u,A),
\end{equation} 
where the inequality follows by the definition of $\mu$.
In particular, \eqref{e:dis2} holds for every Borel  set. Hence,  we decompose $\mu$ in the following way
\[\mu(u,B)= \int_{B} f(x)\, \de x + \int_{J_u \cap B} g(x) \, \de \mathcal{H}^{n-1}(x)\]
for every Borel set $B \subset \Omega$,  for some densities $f$ and~$g$.

We start by calculating the density $f$, which is given by 
\[f(x)=\lim_{r \to 0} \frac{\mu(u,B_r(x))}{\omega_n r^n}\]
for $\mathcal{L}^n$-a.e.~$x \in \Omega$. Let $x \in \Omega$  be such that
\begin{align}
    \label{e:lebpt}& \lim_{r \to 0}\frac{1}{r^n}\int_{B_r(x)} |e(u)(y)-e(u)(x)|^2 \de y =0, \\
    \label{e:nojumppt}&\lim_{r \to 0}\frac{\mathcal{H}^{n-1}(J_u \cap B_r(x))}{r^n}=0\,.
\end{align} 
We recall that $\mathcal{L}^n$-a.e.~$x \in \Om$ satisfy \eqref{e:lebpt} and~\eqref{e:nojumppt}.  Consider $\mathscr{B}_{B_r(x)}$ a  family of covers of~$B_r(x)$ such that  
\begin{align}
\label{e:zetalim}\lim_{r \to 0} &  \frac{\mu(u,B_r(x)) -\sum_{B \in \mathscr{B}_{B_r(x)}} \zeta(u,B)}{r^{n}} = 0 \,,\\
\label{e:leblim}\lim_{r\to 0} & \frac{\mathcal{L}^n(B_r(x))- \sum_{B \in \mathscr{B}_{B_r(x)}}\mathcal{L}^n (  B ) }{r^{n}} = 0 \,.
\end{align}

In view of \eqref{e:zetalim}, for every $\delta>0$ we have
\begin{align}
    f(x) &= \lim_{r \to 0} \frac{\mu(u,B_r(x))}{\omega_n r^n} = \lim_{r \to 0} \frac{\sum_{B \in \mathscr{B}_{B_r(x)}} \zeta(u,B)}{\omega_n r^n} \nonumber
    \\
    &=\lim_{r \to 0} \frac{1}{\omega_n r^n}\sum_{B \in \mathscr{B}_{B_r(x)}} \Bigg (\int_{\R^n} \Big ( \int_B( |(e(u)(y)-e(u)(x)+e(u)(x))\xi \cdot \xi|^2  \de y \nonumber \\
    &\quad+\frac{\pi}{2}\int_{J_u\cap B} |\nu_u(y) \cdot \xi|\,  \de \mathcal{H}^{n-1}(y) \Big )^{p}|\xi|^p e^{-|\xi|^2} \de \xi\Bigg )^{\frac{1}{p}} \nonumber\\
    &\le \lim_{r \to 0} \frac{1}{\omega_n r^n}\sum_{B \in \mathscr{B}_{B_r(x)}} \Big (\int_{\R^n} (1+\delta)^p |e(u)(x) \xi \cdot \xi|^{2p} \mathcal{L}^n(B)^p |\xi|^p e^{-|\xi|^2} \de \xi \Big )^{\frac{1}{p}}\label{e:finalm}\\
    & \quad + \lim_{r \to 0} \frac{1}{\omega_n r^n}\sum_{B \in \mathscr{B}_{B_r(x)}}\bigg (\int_{\R^n} \Big (\int_B \big(1+\frac{1}{\delta}\big) |e(u)(y)-e(u)(x)|^2  \de y \Big )^p |\xi|^{5p} e^{-|\xi|^2} \de \xi\bigg)^{\frac{1}{p}} \label{e:uf}\\
    &\quad+\lim_{r \to 0} \frac{1}{\omega_n r^n}\sum_{B \in \mathscr{B}_{B_r(x)}}\bigg (\int_{\R^n}\Big (\frac{\pi}{2}\int_{J_u\cap B} |\nu_u(y) \cdot \xi|\,  \de \mathcal{H}^{n-1}(y) \Big )^{p} |\xi|^p e^{-|\xi|^2} \de \xi\bigg )^{\frac{1}{p}}\label{e:ancuf}.
\end{align}

The limit in \eqref{e:uf} rewrites as
\begin{align*}
    \big(1+\frac{1}{\delta}\big)
    \bigg (\int_{\R^n}  |\xi|^{5p} e^{-|\xi|^2} \de \xi\bigg)^{\frac{1}{p}}  \lim_{r \to 0} \frac{1}{\omega_n r^n} \sum_{B \in \mathscr{B}_{B_r(x)}}
    \int_B  |e(u)(y)-e(u)(x)|^2  \de y
\end{align*}
which is equal to zero by \eqref{e:lebpt}.
The limit in \eqref{e:ancuf} is bounded by
\begin{align*}
    \frac{\pi}{2} \bigg (\int_{\R^n}  |\xi|^{2p} e^{-|\xi|^2} \de \xi\bigg )^{\frac{1}{p}} \lim_{r \to 0} \frac{1}{\omega_n r^n}\sum_{B \in \mathscr{B}_{B_r(x)}} \mathcal{H}^{n-1}(J_u \cap B)
\end{align*}
which is also equal to zero by \eqref{e:nojumppt}.
Finally, by \eqref{e:leblim} the limit in \eqref{e:finalm} is equal to
\[(1+\delta)\left (\int_{\R^n} \left ( |e(u)(x) \xi \cdot \xi|^2   \right )^{p}|\xi|^p e^{-|\xi|^2} \de \xi\right )^{\frac{1}{p}}=  (1+\delta)\varphi_{p} \big(e(u)(x)\big)\]
for every $\delta >0$,
and thus $f(x)\le \varphi_{p} \big(e(u)(x)\big)$. The other inequality follows from similar arguments.

In order to calculate the density $g$ we consider $x \in J_u$ such that
\begin{align}
    \label{e:jumppt1}& \lim_{r \to 0}\frac{\int_{B_r(x)}|e(u)(y)\xi \cdot \xi|^2\de y}{r^{n-1}} =0,\\
    \label{e:jumppt2}& \lim_{r \to 0} \frac{\mathcal{H}^{n-1}(J_u \cap B_r(x))}{\omega_{n-1} r^{n-1}} =1,\\
    \label{e:jumppt3}& \lim_{r \to 0}\frac{ 1}{\omega_{n-1} r^{n-1}} \int_{B_r(x) \cap J_u} |\nu_u(y) -\nu_u(x)| \, \de \mathcal{H}^{n-1}(y) = 0.
\end{align}
We remark that $\mathcal{H}^{n-1}$-a.e.~$x \in J_u$ satisfy \eqref{e:jumppt1}--\eqref{e:jumppt3}.
We consider a family of covers $\mathscr{B}_{B_r(x)}$ of $B_r(x)$ such that 
\begin{align}
 \label{e:jumppt4}\lim_{r\to0} & \frac{ \mu(u,B_r(x)) - \sum_{B \in \mathscr{B}_{B_r(x)}} \zeta(u,B)}{r^{n-1}} = 0\,.
\end{align}
Thanks to \eqref{e:jumppt4}  we obtain
\begin{align}
    g(x) &= \lim_{r \to 0} \frac{\mu(u,B_r(x))}{\omega_{n-1} r^{n-1}} = \lim_{r \to 0} \frac{\sum_{B \in \mathscr{B}_{B_r(x)}} \zeta(u,B)}{\omega_{n-1} r^{n-1}} \nonumber
    \\
    &=\lim_{r \to 0} \frac{1}{\omega_{n-1} r^{n-1}}\sum_{B \in \mathscr{B}_{B_r(x)}} \Bigg (\int_{\R^n} \Big ( \frac{\pi}{2}\int_{J_u\cap B} |(\nu_u(y)-\nu_u(x) +\nu_u(x)) \cdot \xi|\,  \de \mathcal{H}^{n-1}(y)
     \nonumber \\
    &\quad + \int_B( |e(u)(y)\xi \cdot \xi|^2  \de y \Big )^{p}|\xi|^p e^{-|\xi|^2} \de \xi\Bigg )^{\frac{1}{p}} \nonumber\\
    &  \le \lim_{r \to 0} \frac{1}{\omega_{n-1} r^{n-1}}\sum_{B \in \mathscr{B}_{B_r(x)}}\bigg (\int_{\R^n}\Big (\frac{\pi}{2}\int_{J_u\cap B} |\nu_u(x) \cdot \xi |\,  \de \mathcal{H}^{n-1}(y) \Big )^{p} |\xi|^p e^{-|\xi|^2} \de \xi\bigg )^{\frac{1}{p}} \label{e:ufufufu}\\
    & \quad +\lim_{r \to 0} \frac{1}{\omega_{n-1} r^{n-1}}\sum_{B \in \mathscr{B}_{B_r(x)}}\bigg (\int_{\R^n}\Big (\frac{\pi}{2}\int_{J_u\cap B} |\nu_u(y)-\nu_u(x) |\,  \de \mathcal{H}^{n-1}(y) \Big )^{p} |\xi|^{2p} e^{-|\xi|^2} \de \xi\bigg )^{\frac{1}{p}}\label{e:ancuf2}\\
    &\quad+
    \lim_{r \to 0} \frac{1}{\omega_{n-1} r^{n-1}}\sum_{B \in \mathscr{B}_{B_r(x)}}\bigg (\int_{\R^n} \Big (\int_B |e(u)(y)|^2  \de y \Big )^p |\xi|^{5p} e^{-|\xi|^2} \de \xi\bigg)^{\frac{1}{p}}. \label{e:uf2}
\end{align}
Similarly as before, the limit in \eqref{e:ancuf2} is equal to zero by \eqref{e:jumppt2} and \eqref{e:jumppt3}.
From \eqref{e:jumppt1} it follows that  also \eqref{e:uf2} is zero. 
Finally, we observe that \eqref{e:ufufufu} is equal to
\[\frac{\pi}{2}  \left (\int_{\R^n} |\nu_u(x) \cdot \xi|^{p}|\xi|^p e^{-|\xi|^2} \de \xi\right )^{\frac{1}{p}}=:\beta_{p},\]
hence $g(x) \le \beta_p$. The reverse inequality follows from similar computations.

\end{proof}

\subsection{Proof of Theorem~\ref{compattezza}}


Consider any subsequence $\varepsilon_k \to 0$ and set $u_k:=u_{\varepsilon_k}$. By H\"older's inequality, we have
\[M:= \sup_{k \in \mathbb{N}} \mathcal{F}^1_{\varepsilon_{k}} (u_{k}, \Omega)\le C(n,p) \sup_{k \in \mathbb{N}} \mathcal{F}^p_{\varepsilon_{k}} (u_{k}, \Omega) <\infty.\]
We divide the proof into 6 steps. In Steps 1 and 2 we prove that $u_{k}$ converges pointwise, up to an exceptional set $A$, to some limit function $u$ by a Fr\'echet-Kolmogorov argument. Steps 3--6 are devoted to show that $u$ is indeed a $\text{GSBV}^{\mathcal{E}}$-function and that~$A$ is of finite perimeter.

\noindent \textbf{Step 1:}  
Fix two open sets $F \Subset E \Subset \Omega$ and define $f_{k}\colon \Omega \times \R^n \to \R$ by
\[
f_{k}(x,\xi):=\tau(u_{k}(x)\cdot \xi):=\arctan(u_{k}(x)\cdot \xi).
\]
Let $k \in \N$ be sufficiently large so that $ B_4(0)\setminus B_{1/4}(0) \subset \frac{\Omega-\Omega}{\varepsilon_k}$ (notice that $\Omega - \Omega$ is an open set containing the origin)  and $4\varepsilon_{k} \leq {\rm dist}(E, \partial \Om)$. In particular, this implies that  $F \subset E \subset \Omega -\varepsilon_k t \xi$ for every $\xi \in  B_{4}(0)\setminus B_{1/4}(0) $ and every $t \in [0, 1]$.  For simplicity of notation, let us set $R := B_{2}(0) \setminus B_{1/2}(0)$.  In order to apply Fr\'echet-Kolmogorov Theorem on $F \times R$, we start by showing that for every $\alpha >0$ there exist $\bar t\in(0,1)$ such that
\begin{equation}\label{equicont}
    \int_{F \times R} |f_{k}(x+ \eta,\xi)-f_{k}(x,\xi)|\, \de x \, \de \xi \le \alpha \qquad \text{for $\eta \in B_{\bar t}(0)$.}
\end{equation}

 Instead of proving \eqref{equicont}, we prove a slightly different equivalent statement: for every $\alpha >0$ there exist $\bar t\in(0,1)$ and $\bar k \in \N$ such that
\begin{equation}\label{equicont72}
    \int_{F \times R} |f_{k}(x+ \eta,\xi)-f_{k}(x,\xi)|\, \de x \, \de \xi \le \alpha \qquad \text{for $\eta \in B_{\bar t}(0)$ and $k \geq \overline{k}$.}
\end{equation}

For every $\eta \in B_1(0)$, $\xi \in \R^n$ and $j =1,2,\dotsc$, we define $\xi_\eta^j \in \R^n$ as 
\begin{equation*}
    \xi_\eta^j:= \xi+ \frac{1}{j} \eta.
\end{equation*}
We remark that 
by Lemma~\ref{l:fundamentalest} for every $\eta \in B_1(0)$ and $j \in \N$ we have that  
\begin{align}
\label{e:nstima8}
&\sup_{e \in \mathbb{R}^n}\int_{R} \frac{|e|}{1+|e\cdot \xi_\eta^j|^2} \, \de\xi 
\leq  \sup_{e \in \mathbb{R}^n }\int_{B_4(0)\setminus B_{1/4}(0)} \frac{|e|}{1+|e \cdot \xi|^2}  \, \de\xi  \le C
\end{align}
for some positive constant $C$ only depending on the dimension $n$.

Let us now fix  $\overline{j} \ge 4$ such that $\overline{j} \ge 4|F|C/\alpha$.
We repeatedly apply the triangle inequality to obtain  for $t \in [0, 1]$ 
\begin{align}
\label{291}
     &\left |f_{k}(x+ t \eta , \xi ) - f_{k}( x , \xi ) \right | \\
     &\quad \le |f_{k}( x + t \eta , \xi )- f_{k}( x + t \eta ,  \xi_\eta^{\overline{j}} )|+ |f_{k}( x + t \eta ,  \xi_\eta^{\overline{j}}  ) - f_{k} (x -  \overline{j}  t \xi,  \xi_\eta^{\overline{j}} )| \nonumber\\
     &\qquad +|f_{k}( x -  \overline{j}  t\xi, \xi_\eta^{\overline{j}}  ) - f_{k}( x - \overline{j}  t \xi , \xi ) | + | f_{k} ( x -  \overline{j}t \xi , \xi ) - f_{k} ( x , \xi ) |. \nonumber
\end{align}

The first term in \eqref{291} is bounded as follows
\begin{align}
\label{stima xi}
   &   |f_{k}( x + t \eta , \xi )- f_{k}( x + t \eta , \xi_\eta^{\overline{j}} )|
    \\
    &
    \qquad = |\tau(u_k(x + t \eta ) \cdot  \xi  ) - \tau ( u_k( x + t \eta ) \cdot  \xi_\eta^{\overline{j}}   ) | = \left |\int_{u_k ( x + t \eta )\cdot \xi}^{u_k( x + t \eta )\cdot   \xi_\eta^{\overline{j}}   } \frac{\de s}{1+s^2} \right | \nonumber\\
    &\qquad\le  \max \Big \{ \frac{1}{1+|u_k(x+t\eta)\cdot \xi|^2},\frac{1}{1+|u_k(x+t\eta)\cdot  \xi_\eta^{\overline{j}}   |^2}  \Big \} \Big |u_k(x+t\eta)\cdot (  \xi_\eta^{\overline{j}}   -  \xi  ) \Big |\nonumber \\
     &\qquad \le \frac{1}{\overline{j}} \max \Big \{ \frac{1}{1+|u_k(x+t\eta)\cdot \xi|^2},\frac{1}{1+|u_k(x+t\eta)\cdot \xi_\eta^{\overline{j}}|^2}  \Big \} |u_k(x+t\eta)|. \nonumber
\end{align}
 Thanks to~\eqref{e:nstima8}, we deduce from~\eqref{stima xi} that
\begin{equation}
\label{e:nstima2}
\begin{split}
\int_{R} |f_{k}( x + t \eta , \xi )- f_{k}( x + t \eta , \xi_\eta^{\overline{j}} )| \, \de \xi
\le \frac{1}{\overline{j}}C  \, . 
\end{split}
\end{equation}
for the same constant~$C$ defined in~\eqref{e:nstima8}.
Hence, we infer from~\eqref{e:nstima2} and from the choice of~$\overline{j}$ that for every $x \in F$,  $t \in [0, 1]$, $k \in \mathbb{N}$, and $\eta \in B_1(0)$, it holds
 \begin{equation}
     \label{e:nstima3}
    \int_{R} | f_{k}(x + t\eta, \xi) - f_{k} (x + t\eta, \xi^{\overline{j}}_{\eta})  | \, \de \xi \leq \frac{\alpha}{4|F|}.
 \end{equation}
 The very same argument allows us to bound the third term  on the right-hand side of~\eqref{291}  for every $x \in F$,  $t \in [0, 1]$, $k \in \mathbb{N}$, and $\eta \in B_1(0)$
 as 
 \begin{equation}
     \label{e:nstima4}
    \int_{R}  | f_{k} ( x -\overline{j}  t \xi ,  \xi_\eta^{\overline{j}} ) - f_{k} ( x -  \overline{j} t \xi ,   \xi )|  \, \de \xi \leq \frac{\alpha}{4|F|}.
 \end{equation}

For the second term  on the right-hand side of~\eqref{291}  we have
\begin{align*}
    &\int_{R}  \int_{F}| f_{k} ( x + t \eta ,  \xi_\eta^{\overline{j}}  ) - f_{k} ( x -  \overline{j} t \xi ,  \xi_\eta^{\overline{j}}  ) | \de x \, \de \xi\\
    &= \int_{R}  \int_{F -  \overline{j}  t\xi}|f_{k} ( x + t  \overline{j} \xi_\eta^{\overline{j}} , \xi_\eta^{\overline{j}}  ) -f_{k} ( x ,  \xi_\eta^{\overline{j} } ) | \de x \,\de \xi\\
    &\le\int_{E \times R} |\tau(u_{k} ( x + t  \overline{j} \xi_\eta^{\overline{j}}  )\cdot   \xi_\eta^{\overline{j}}   ) -\tau (u_{k} (x) \cdot   \xi_\eta^{\overline{j}}   ) | \de x \, \de \xi,
\end{align*}
if $ \overline{j}  t \ll 1$,  since $F \Subset E$. We apply the change of variable  $\xi = \xi^{\overline{j}}_{\eta}$  to the last integral above to obtain 
\begin{align*}
 \int_{R}  \int_{F} &| f_{k} ( x + t \eta ,  \xi_\eta^{\overline{j}}   ) - f_{k} ( x -   \overline{j} t \xi ,   \xi_\eta^{\overline{j}}   ) | \de x \, \de \xi\\
 &\leq \int_{E \times (B_{4}(0)\setminus B_{1/4}(0))} |\tau(u_{k}( x + t   \overline{j}   \xi ) \cdot  \xi  ) - \tau ( u_{k} (x) \cdot \xi ) |  \de x \, \de \xi\,.
\end{align*}

Recall that, by Lemma~\ref{lemmaG5.1}, for $t \in [0,1)$ and $\xi \in B_{4}(0)\setminus B_{1/4}(0)$ such that $E + t \overline{j} \xi \subset \Om$ we have
\begin{equation}\label{6555}
    \int_{E} |\tau(u_{k}(x + t   \overline{j}   \xi ) \cdot  \xi ) - \tau ( u_{k}(x) \cdot \xi ) | \de x \le C_{E} t  \overline{j}   (1 + F_{ t   \overline{j}  ,\xi}(u_k,  E) ) .
\end{equation}
Choose  $t_{\alpha} \in (0, 1) $ sufficiently small  so that 
\begin{align}
     & \label{e:one-inclusion} F \subset E + \overline{j} t \xi \subset \Om \qquad \text{for every $\xi \in  (B_{4}(0)\setminus B_{1/4}(0))$  and every $t \in [0, t_{\alpha}]$,}\\
     & C_E    \bar j t_{\alpha} (1+M)\le \frac{\alpha}{4}\nonumber.
\end{align} 
Let $ \bar k\in \N$ sufficiently large  and let us set $t_k:=  i_k \varepsilon_k \in (\frac{t_{\alpha}}{2} , t_{\alpha})$
for $k \ge \bar k$ and for some $i_k\in \N$. Thanks to~\eqref{e:one-inclusion}, we apply~\eqref{6555} for $t = t_{k}$ and Lemma~\ref{Gob5.3} with the choice $m = \overline{j} i_{k}$, obtaining 
\begin{align*}
    \int_{E  \times  (B_{4}(0)\setminus B_{1/4}(0)) } &|\tau(u_{k}(x +  \overline{j}  t_k \xi ) \cdot  \xi  ) - \tau ( u_{k} (x) \cdot  \xi  ) | \de x \,\de \xi\\
    &=\int_{E  \times (B_{4}(0)\setminus B_{1/4}(0)) } | \tau ( u_{k} ( x +   \overline{j}   i_k \varepsilon_k\xi)\cdot \xi) - \tau ( u_{k}(x) \cdot \xi)| \de x \,\de \xi\\
    &\le C_{E}   \overline{j}   i_k\varepsilon_k \int_{  (B_{4}(0)\setminus B_{1/4}(0)) } ( 1 + F_{   \overline{j}   i_k \varepsilon_k,\xi}(u_k,  E )) \de \xi
    \\
    & \vphantom{\int_{E}}\le C_E   \overline{j}   i_k \varepsilon_k(1+\mathcal{F}^1_{\varepsilon_k}(u_k,  \Omega)).
\end{align*}
Summarizing we have thus obtained for every $k \ge \bar k$  and every $\eta \in B_{1}(0)$ 
\begin{align}
    \label{e:nstima5}
    \int_{F \times R}  |f_{k}(x + t_k \eta ,   \xi_\eta^{\overline{j}} ) - f_{k}(x-   \overline{j}  t_k\xi,  \xi_\eta^{\overline{j}}  ) | \de x \, \de \xi & \leq  C_E    \overline{j}  i_k \varepsilon_k(1+\mathcal{F}^1_{\varepsilon_k}(u_k,  \Omega))
    \\
    &
     \leq C_E\overline{j}  t_{\alpha} ( 1 + M) \leq  \frac{\alpha}{4}. \nonumber
\end{align}
From the very same argument, we infer that for every $k \ge \bar k$  and every $\eta \in B_{1}(0)$ 
\begin{equation}
    \label{e:nstima6}
    \int_{F \times R}  |f_{k}(x -   \overline{j}   t_k\xi,\xi) - f_{k} ( x , \xi ) | \de x \, \de \xi \leq C_E   \overline{j}   i_k\varepsilon_k(1+\mathcal{F}^1_{\varepsilon_k}(u_k,  \Omega)) \leq \frac{\alpha}{4} \,.
\end{equation}

Combining~\eqref{e:nstima3}, \eqref{e:nstima4}, \eqref{e:nstima5}, \eqref{e:nstima6}, we have shown that for every $k \geq \bar{k}$ and every $\eta \in B_{1}(0)$ it holds
\begin{equation}
\label{e:interfk}
    \int_{F \times R } f_{k} (x + t_{k} \eta, \xi) - f(x , \xi) \, \di x \, \di \xi \leq \alpha\,,
\end{equation}
where we recall that $t_{k} \in (\frac{t_{\alpha}}{2}, t_{\alpha})$. Finally, setting $\bar t := t_{\alpha}/2$, \eqref{e:interfk} yields 
\begin{equation*}
    \int_{R}  \int_{F}|f_{k}(x+\eta,\xi)-f_{k}(x,\xi)| \de x \, \de  \xi  \leq \alpha, 
\end{equation*}
for every $k \ge \overline{k}$ and every $\eta \in B_{\overline{t}} (0)$, which is precisely \eqref{equicont72}.

\noindent\textbf{Step 2:}  Now we prove that for every $\alpha >0$ there  exists $\bar t \in (0, 1)$  such that
\begin{equation}\label{equicont2}
    \int_{F \times R} |f_{k}(x,\xi + \eta)-f_{k}(x,\xi)| \,\de x \,\de  \xi  \le \alpha \qquad \text{for $\eta \in B_{\bar t}(0)$ and $k \in \mathbb{N}$.}
\end{equation}
Fix $\eta \in \Ss^{n-1}$ and let $t \in (0,1).$
By arguing as in \eqref{stima xi} we have that for $x \in F$
\begin{align*}
   & |\tau(u_k(x) \cdot (\xi +t\eta)) - \tau(u_k(x) \cdot \xi)| \\
    &\leq t \, \max \Big \{ \frac{1}{1+|u_k(x)\cdot \xi|^2},\frac{1}{1+|u_k(x)\cdot (\xi + t \eta)|^2}  \Big \} |u_k(x) \cdot\eta|.
\end{align*}
From this last inequality  we deduce from Lemma~\ref{l:fundamentalest} that
\begin{equation}
\label{e:nstima10}
    \int_{R}|\tau(u_k(x) \cdot (\xi +t\eta)) - \tau(u_k(x) \cdot \xi)|\,\de  \xi  \leq C t\,,
\end{equation}
for a positive constant $C$ only depending on the dimension~$n$.  
By \eqref{e:nstima10} there exists $\bar{t} \in (0, 1)$ such that for every $x \in F$, $k \in \mathbb{N}$, $\eta \in \mathbb{S}^{n-1}$, and every $t \in [0, \bar{t}]$ 
\begin{equation}
\label{e:nstima12}
    \int_{R} |f_{k} (x, \xi + t\eta)  - f_{k} (x, \xi) | \, \de  \xi  \leq \alpha.
\end{equation}
 Eventually, \eqref{e:nstima12} leads to~\eqref{equicont2}.

 Thanks to~\eqref{equicont} and~\eqref{equicont2}, it is not difficult to see that we are in position to apply  Fr\'echet-Kolmogorov Theorem (see, e.g.,~\cite[Theorem~4.26]{Brezis})  and obtain that the sequence $\{f_k\}_{k \in \N}$ is relatively compact in  $L^1(F \times R)$  for every open set $F$ compactly contained in $\Omega$. Therefore, up to passing to a subsequence, we have $f_k \to f_{\infty}$ as $k \to \infty$ strongly in  $L_{loc}^1(\Omega \times R)$. By a diagonal argument we can possibly pass to another subsequence such that 
 \begin{equation}
 \label{e:pointconv1}
 f_k \to f_{\infty} \qquad \text{pointwise a.e.~in $\Omega \times R$}. 
\end{equation}
 From \eqref{e:pointconv1},  we deduce the existence of an orthonormal basis $\{\eta_1,\ldots,\eta_n\}$ of $\R^n$ such that 
\begin{equation*}
 \tau(u_k \cdot \eta_i) \to f^i_\infty \qquad \text{pointwise a.e.~in $\Omega$ for every $i=1,\dotsc,n$ as $k \to \infty$},
\end{equation*}
for some measurable function $f^i_\infty \colon \Omega \to [-\pi/2,\pi/2]$. Since $\arctan$ is a diffeomorphism with its image, we deduce the existence of a measurable function $u\colon \Omega \to \R^n$ such that  $ u=0$ on~$A$ and
\begin{align}
    \label{e:pointconv5}
    u_k \to u \qquad  &\text{a.e.~in $\Omega \setminus A$ as $k \to \infty$}\\
     \label{e:pointconv6}
    |u_k| \to \infty \qquad   &\text{a.e.~in $A$ as $k \to \infty$},
\end{align}
where  $A$ is the set defined in~\eqref{e:setA}. Eventually, up to passing to a subsequence, we also assume that $u_k \to u$ in measure in $\Omega \setminus A$.  We further remark that
\begin{equation}\label{9189}
    x \in A \iff \lim_{k \to  \infty } |u_k(x)\cdot \xi|=\infty \quad \text{for a.e. } \xi \in \R^n.
\end{equation} 

\noindent \textbf{Step 3:}
Let $\lambda >0 $ and let $\tau_\lambda\colon \R \to [-\lambda,\lambda]$ be a strictly increasing function such that  $\tau'_\lambda \leq 1$, $ \lim_{t \to \pm \infty} \tau_\lambda(t)=\pm \lambda$, and $\tau_{\lambda}(t)=t$ for $t \in [-\lambda/2,\lambda/2]$.
Fix $\xi \in \R^n$,  and let $y \in \Pi^{\xi}$. For every $\lambda>0$, we want to construct suitable modifications $\{v_{j,\lambda}\}_{j \in \N} \subset \text{SBV}(\Omega_y^\xi)$ of $(\hat u_{k})_{y,\lambda}^{\xi}(t):=\tau_\lambda(u_k(y+t \xi)\cdot \xi)$, as in~\cite[Theorem~3.4, Step $2$]{Gobbino},  such that
\begin{equation}\label{3.27G}
    v_{j,\lambda} \to \hat u_{y,\lambda}^{\xi} \text{ in } L^1_{loc}(\Omega_y^\xi)\text{ as } j \to \infty ,  
\end{equation}
and
\begin{equation}\label{3.28G}
    \liminf_{k \to  \infty } F_{\varepsilon_k}\big((\hat u_{k})_{y,\lambda}^{\xi},I \cap (I-\varepsilon_k)\big ) \ge \frac{\pi}{2}\mathcal{MS}_{\frac{2}{\pi}}(v_{j,\lambda},I) 
\end{equation}
 for every $j$ sufficiently large and every interval $I \Subset \Omega_y^\xi$,
where $\hat u_{y,\lambda}^{\xi}(t):=\tau_\lambda(u(y+t\xi)\cdot \xi)$ for $y+t\xi \in \Omega \setminus A$ and $\hat u_{y,\lambda}^{\xi}(t) = \pm \lambda$ for $y+t \xi \in A$. 

We notice that by applying Fubini's Theorem for a.e. $\xi$ we have
\[\mathcal{L}^n(\{x \in A: \liminf_{k \to \infty} |u_k(x)\cdot \xi|<\infty\})=0.\]
Then thanks to \eqref{e:pointconv5}--\eqref{e:pointconv6}, up to a subsequence which does not depend on $\xi$ and $y$ we have (for every $\lambda >0$)
\begin{align}
\label{e:pointconv7}
    &(\hat u_{k})_{y,\lambda}^{\xi} \to \hat u_{y,\lambda}^{\xi} \qquad \text{in $L^1_{loc}((\Omega \setminus A)_y^\xi)$ for every $\xi$ and for $\mathcal{H}^n$-a.e.~$y \in \Pi^\xi$},\\
 \label{e:pointconv8}   
    &|(\hat u_{k})_{y,\lambda}^{\xi}| \to |\hat u_{y,\lambda}^{\xi}| =\lambda \qquad  \text{in $L^1_{loc}(A_y^\xi)$ for a.e.~$\xi$, for $\mathcal{H}^{n-1}$-a.e.~$y \in \Pi^\xi$,}
\end{align}
as $k \to \infty$. We further observe that, since $0  \leq  \tau'_\lambda \leq 1$, then $|\tau_\lambda(t)- \tau_\lambda(t')| \leq |t -t'|$ for every $t,t' \in \mathbb{R}$, leading to
\begin{equation}
\label{e:monotonicitylambda}
    F_\varepsilon(f,B) \geq F_\varepsilon(\tau_\lambda(f),B), 
\end{equation}
for every $\varepsilon >0$, every measurable set $B \subset \mathbb{R}$, and every measurable function $f \colon B \to \mathbb{R}$.

We extend the function $\hat u^\xi_{y,\lambda}$ to $\R$ in such a way that $\hat u^\xi_{y,\lambda}(t)=0$ for $t \in \R \setminus \Omega_y^\xi$.
Let $a \in \R$  satisfy  conditions $(i)$ and $(ii)$ of Lemma~\ref{Lemma3.3} for $\hat u^\xi_{y,\lambda}$, and let $I_j^z:=[a+\frac{z}{j},a+\frac{z+1}{j}]\Subset \R.$ We define $v_{j,\lambda}$ in every interval $I_j^z$ in the following way: 
\begin{itemize}
    \item If $j(\hat u_{y,\lambda}^{\xi} (a+\frac{z+1}{j})-\hat u_{y,\lambda}^{\xi}(a+\frac{z}{j}))^2\le \frac{\pi}{2},$ then $v_{j,\lambda}$ is the affine function that coincides with $\hat u_{y,\lambda}^{\xi}$ at the endpoints of $I_j^z;$
    \item if $j(\hat u_{y,\lambda}^{\xi}(a+\frac{z+1}{j})-\hat u_{y,\lambda}^{\xi}(a+\frac{z}{j}))^2> \frac{\pi}{2},$ then $v_{j,\lambda}$ is the piecewise constant function that coincides with $\hat u_{y,\lambda}^{\xi}$ at the endpoints of $I_j^z$ and has a unique discontinuity in the middle  point of the interval.
\end{itemize}
Then by Lemma~\ref{Lemma3.3}, up to subsequences, \eqref{3.27G} holds. Moreover, $v_{j,\lambda} \in \text{SBV}(\Omega_y^\xi)$ and
\begin{equation*}
    \frac{\pi}{2} \mathcal{MS}_{\frac{2}{\pi}}(v_{j;\lambda},I_j^z)=\min \left \{ \frac{\pi}{2},j \left ( \hat u_{y,\lambda}^{\xi}\left ( a+\frac{z+1}{j}\right )-\hat u_{y,\lambda}^{\xi}\left ( a+\frac{z}{j}\right )\right )^2 \right \}
\end{equation*}
for every $j \in \N$ and $z \in \Z$ such that $I_j^z \Subset \Omega_y^\xi$. 

Let $I \subset \Omega_y^\xi$ be an interval and $F \Subset I \cap (I-\varepsilon_k)$ for $k$ sufficiently large, note that such $F$ exists since $I \cap (I -\varepsilon_k) \to I$ as $k \to \infty$.
 We notice that for every $j$ sufficiently large, depending on $F$, every intervals $I_j^z$ such that $I_j^z \subset I \cap (I-\varepsilon_k)$ also satisfy $F \subset \cup_{z} I_j^z$. By virtue of \eqref{e:pointconv7} and \eqref{e:pointconv8} we can make use of Lemma~\ref{Lemma3.2} in $I_j^z$ and infer, for a.e. $\xi \in \R^n$ and for $\mathcal{H}^{n-1}$-a.e. $y \in \Pi^\xi$, that
\begin{equation*}
    \liminf_{k \to  \infty } F_{\varepsilon_k}((\hat u_{k})_{y,\lambda}^{\xi},I_j^z) \ge \frac{\pi}{2} \mathcal{MS}_{\frac{2}{\pi}}(v_{j,\lambda},I_j^z).
\end{equation*}
To prove \eqref{3.28G} we sum over all $z$ such that $I_j^z \subset I \cap (I-\varepsilon_k)$ for every $j$ sufficiently large 
\begin{equation*}
    \liminf_{k \to  \infty } F_{\varepsilon_k}((\hat u_{k})_{y,\lambda}^{\xi},I \cap (I-\varepsilon_k)) \ge \frac{\pi}{2}\mathcal{MS}_{\frac{2}{\pi}}(v_{j,\lambda},F).
\end{equation*}
From the semicontinuity of the Mumford-Shah functional with respect to the $L_{loc}^1$-convergence, taking the limit as $j \to \infty$ and then as $F \nearrow I$, we obtain 
\begin{equation}\label{semitilde}
    \liminf_{k \to  \infty } F_{\varepsilon_k}\big((\hat u_{k})_{y,\lambda}^{\xi},I \cap (I- \varepsilon_k)\big) \ge \frac{\pi}{2} \mathcal{MS}_{\frac{2}{\pi}}(\hat u_{y,\lambda}^{\xi},I).
\end{equation}

We recall that
\begin{align*}
\mathcal{F}^1_{\varepsilon}(u,\Omega) & = \sup_{\mathscr{B}}\sum_{B \in \mathscr{B}}\int_{\frac{\Omega -\Omega}{\varepsilon}} F_{\varepsilon,\xi}(u,B) e^{-|\xi|^2} \de \xi
\\
&
=\sup_{\mathscr{B}}\sum_{B \in \mathscr{B}}\int_{\frac{\Omega-\Omega}{\varepsilon}} \left (\int_{\Pi^\xi}F_{\varepsilon}(\hat u_y^\xi,B_y^\xi\cap (B-\varepsilon\xi)_y^\xi) \de \mathcal{H}^{n-1}(y) \right ) |\xi| e^{-|\xi|^2} \de \xi.
\end{align*}
From \eqref{e:monotonicitylambda},  Fatou's Lemma, and \eqref{energia_limitata}, we get
\begin{equation*}
\begin{split}
    \sup_{\mathscr{B}}&\sum_{B \in \mathscr{B}}\int_{\R^n} \int_{\Pi^\xi}  \liminf_{k \to  \infty } F_{\varepsilon_k}((\hat{u}_k)_{y,\lambda}^\xi,B_y^\xi\cap (B-\varepsilon_k\xi)_y^\xi) |\xi| e^{-|\xi|^2}\de \mathcal{H}^{n-1}(y)\, \de \xi
    \\
     &
     \le \sup_{\mathscr{B}}\sum_{B \in \mathscr{B}}\int_{\R^n} \int_{\Pi^\xi} \liminf_{k \to  \infty } F_{\varepsilon_k}((\hat{u}_k)_{y}^\xi,B_y^\xi\cap (B-\varepsilon_k\xi)_y^\xi)|\xi| e^{-|\xi|^2} \de \mathcal{H}^{n-1}(y) \,\de \xi
     \\
     &
     \le\liminf_{k \to  \infty } \sup_{\mathscr{B}}\sum_{B \in \mathscr{B}} \int_{\frac{\Omega -\Omega}{\varepsilon_k}} \int_{\Pi^\xi} F_{\varepsilon_k}((\hat{u}_k)_{y}^\xi,B_y^\xi\cap (B-\varepsilon_k\xi)_y^\xi) |\xi| e^{-|\xi|^2}\de \mathcal{H}^{n-1}(y)\, \de \xi 
     \\
    &
    \le  \liminf_{k \to \infty}\mathcal{F}^1_{\varepsilon_k}(u_k,\Omega) \le M,
\end{split}
\end{equation*}
which yields for every open ball $B \subset \Omega$
\begin{equation}\label{bound1}
    \!\!\! \liminf_{k \to  \infty } F_{\varepsilon_k}((\hat{u}_k)_{y,\lambda}^\xi,B_y^\xi\cap (B-\varepsilon_k\xi)_y^\xi) \le C \quad \text{for a.e. } \xi \in \R^n, \text{  for $\mathcal{H}^{n-1}$-a.e.~$y \in \Pi^\xi$}
\end{equation}
where the constant $C>0$ may depend on $\xi$ and $y$ but not on $\lambda$ and $B$.
The above inequality combined with \eqref{semitilde} implies that $\hat u_{y,\lambda}^\xi \in \text{SBV}_{loc}(\Omega_y^\xi)$ for a.e. $\xi \in \R^n$ and $\mathcal{H}^{n-1}$-a.e. $y \in \Pi^\xi$. 

\noindent\textbf{Step 4:}
We claim that for a.e.~$\xi \in \R^n$, for $\mathcal{H}^{n-1}$-a.e.~$y \in \Pi^\xi$, and for every bounded interval $I \subset \Omega_y^\xi$ we have
\begin{equation}
\label{3831}
   \mathcal{H}^0(A_y^\xi\cap I) \ne 0 \implies \mathcal{H}^0(J_{\hat u_{y,\lambda}^{\xi}}\cap I) \ne 0 \text{ or } I \subset A_y^\xi.
\end{equation}

If $I\subset A_y^\xi$ we  have nothing to prove.  Let us assume that there exists  $s \in I \setminus A_y^\xi.$ By contradiction, assume $\mathcal{H}^0(J_{\hat u_{y,\lambda}^{\xi}}\cap I)=0.$ For a.e.~$\xi \in \R^n$, for $\mathcal{H}^{n-1}$-a.e.~$y \in \Pi^\xi$, by \eqref{semitilde} and \eqref{bound1} we get
\begin{equation}
\label{4163}
    \int_{I} |\nabla \hat u_{y,\lambda}^{\xi}(t)|^2 \, \de t \le C .
\end{equation}
Since $\hat u_{y,\lambda}^{\xi}$ belongs to $ \text{SBV}(I)$ and since $\mathcal{H}^0(J_{\hat u_{y,\lambda}^{\xi}}\cap I)=0$,  it is  absolutely continuous in~$I$. Hence,  for $\bar t \in I$  we have
\begin{align*}
    |\hat u_{y,\lambda}^{\xi}(\bar t)|&=\left |\hat u_{y,\lambda}^{\xi}(s)+\int^t_s \nabla \hat u_{y,\lambda}^{\xi}(t) \, \de t \right |\\
    &\le  |\hat u_{y,\lambda}^{\xi}(s)|+  \bigg| \int^t_s  |\nabla \hat u_{y,\lambda}^{\xi}(t)| \,\de t \bigg|  \\
    &\le  |\hat u_{y,\lambda}^{\xi}(s)|+\text{lenght}(I)^{1/2}  \left |\int^t_s  |\nabla \hat u_{y,\lambda}^{\xi}(t)|^2 \,\de t \right|^{1/2} \le C,
\end{align*}
for every $\bar t \in I$, where the last inequality follows from \eqref{4163} and the fact that $s \in I \setminus A_y^\xi$. 
Since $\mathcal{H}^0(A_y^\xi\cap I) \ne  0 $ and \eqref{9189} holds, there exists $\bar t \in A_y^\xi\cap I $ such that $|\hat u_{y,\lambda}^{\xi}(\bar t)|>C$ for every $\lambda >0$ sufficiently large, which is a contradiction. This concludes the proof of~\eqref{3831}.

\noindent\textbf{Step 5:}
Let $I \subset \Omega_y^\xi$ be a bounded interval.
We claim that, for a.e. $\xi \in \R^n$ and for $\mathcal{H}^{n-1}$-a.e. $y \in \Pi^\xi$, the set $A_y^\xi$ is a finite union of intervals and
\begin{equation}\label{4231}
    \frac{\pi}{2}\mathcal{H}^0(\partial^* A_y^\xi \cap I)\le \liminf_{k \to  \infty } F_{\varepsilon_k}((\hat{u}_k)_{y,\lambda}^\xi,I\cap (I-\varepsilon_k)).
\end{equation}
Moreover, $A$ is a set of finite perimeter.

Denote by $(A_y^\xi)^d$ the points of density $d$ of $A_y^\xi.$ We want to prove that $(A_y^\xi)^1$ is an open set. By contradiction, suppose that there exist $t \in (A_y^\xi)^1$ and a sequence $\{s_j\}_{j \in \N} \subset \Omega_y^\xi \setminus (A_y^\xi)^1$ such that $s_j \to t.$ Since almost every points of $\Omega_y^\xi \setminus (A_y^\xi)^1$ belong to $(A_y^\xi)^0$, we find another sequence $\{t_j\}_{j \in \N} \subset (A_y^\xi)^0$ such that $t_j   \to t$. Without loss of generality we  assume $t^0_j< t^0_{j+1}$ for every $j \in \N.$ By definition of points of density zero, for every $j$, there exists a point $t^1_j \in [t^0_j,t^0_{j+1})$ of density one. By the previous step, there exists another increasing sequence $t_j \in [t^0_j,t^0_{j+1}) \cap J_{\hat{u}_{y,\lambda}^{\xi}}$.
Since the sequence $\{t_j\}_{j \in \N}$ is strictly increasing, there exist disjoint open intervals $I_j$ such that $t_j \in I_j$. 
Now, by applying \eqref{semitilde} and \eqref{bound1} on every $I_j$  we obtain  $\sum_j 1\le \sum_j \mathcal{H}^0(J_{\hat{u}_{y,\lambda}^{\xi}} \cap I_j)\le C,$ which is a contradiction. Hence, $(A^{\xi}_{y})^{1}$ is open.

From now on we assume that $A_y^\xi$ coincides with $(A_y^\xi)^1$.
Suppose that $|I \setminus A_y^{\xi}|>0$, otherwise there is nothing to prove. Let $K$ be a connected component of $A_y^\xi \cap I$, and let $t_1 \in K$, $t_2 \in I \setminus A_y^\xi.$ Then, by Step 4, there exists $t \in J_{\hat u_{y,\lambda}^{\xi}} \cap I$ between $t_1$ and $t_2$.
Using a similar argument as above,  we prove that $\mathcal{H}^0(J_{\hat u_{y,\lambda}^{\xi}} \cap I)$ is greater than or equal to the number of connected components of~$A_y^\xi \cap I$ and is uniformly bounded, from which it follows that~$A_y^\xi$ is the union of a finite number of intervals.
Finally,~\eqref{4231} is a consequence of $\mathcal{H}^0(\partial^*A_y^\xi \cap I)\le \mathcal{H}^0(J_{\hat{u}_{y,\lambda}^\xi} \cap I)$  and of~\eqref{semitilde}.

To conclude we need to prove that $A$ is of finite perimeter.
Multiplying \eqref{4231} by $|\xi| e^{-|\xi|^2}$ and then integrating in $\xi \in \R^n$ and $y \in \Pi^\xi$ we obtain for every ball $B \subset \Omega$
\begin{align*}
C(n) \mathcal{H}^{n-1}(\partial^* A \cap B) &\le \int_{\R^n} \Big ( \int_{\Pi^\xi} \mathcal{H}^0(\partial^* A_y^\xi \cap B_y^\xi) \de \mathcal{H}^{n-1}(y) \Big ) |\xi| e^{-|\xi|^2}\de \xi \\
&\le \int_{\R^n} \Big (\int_{\Pi^\xi} \liminf_{k \to \infty} F_{\varepsilon_k}((\hat u_k)_{y,\lambda}^\xi,B_y^\xi \cap B_y^\xi-\varepsilon_k) \de \mathcal{H}^{n-1}(y) \Big ) |\xi| e^{-|\xi|^2} \de \xi \\
&\le \liminf_{k \to \infty} \mathcal{F}^1_{\varepsilon_k}(u_k,B) ,
\end{align*}
where the last inequality follows from Fatou's Lemma. This proves in particular that $A$ has locally finite perimeter in $\Omega$. By Besicovitch covering Theorem there exist a dimensional constant $j(n)$ and $\mathcal{B}_1, \ldots, \mathcal{B}_{j(n)}$ countable families of pairwise disjoint open balls contained in $\Omega$, such that $\Omega \subset \cup_{i=1}^{j(n)} \cup_{B \in \mathcal{B}_i} B$. 
\begin{align*}
C(n)\mathcal{H}^{n-1}(\partial^* A) &\le \sum_{i=1}^{j(n)} \sum_{B \in \mathcal{B}_i} C(n) \mathcal{H}^{n-1}(\partial^* A \cap B) \\
&\le \sum_{i=1}^{j(n)} \sum_{B \in \mathcal{B}_i}\liminf_{k \to \infty} \mathcal{F}^1_{\varepsilon_k}(u_k,B) \\
&\le j(n) \liminf_{k \to \infty} \sum_{B \in \mathcal{B}_i}\mathcal{F}^1_{\varepsilon_k}(u_k,B) \\
&\leq j(n) \liminf_{k \to \infty}\mathcal{F}^1_{\varepsilon_k}(u_k,\Om) 
<\infty,
\end{align*}
where, in the last but one inequality above, we used Remark \ref{r:finnum}.

\noindent\textbf{Step 6:} 
We conclude by proving that  $u \in \text{GSBD} (\Omega)$  and  that~\eqref{semi378} holds.

Fix $\mathscr{B}$ be any finite family of disjoint balls in $\Omega$ and let $B \in \mathscr{B}$.
For a.e. $\xi \in \R^n$, for $\mathcal{H}^{n-1}$-a.e. $y \in \Pi^\xi$, for every interval $I \subset  B_y^\xi$ of the form $I=(t-\delta,t+\delta)$ for $t \in \partial^* A_y^\xi$ and $\delta >0$, by  \eqref{e:monotonicitylambda},  \eqref{semitilde} and  \eqref{4231}, we have
\begin{align*}
     \liminf_{k \to  \infty} &\ F_{\varepsilon_k}((\hat{u}_k)_y^\xi,B_y^\xi \cap (B-\varepsilon_k\xi)_y^\xi)
     \\
     &
     \ge \liminf_{k \to  \infty } F_{\varepsilon_k}((\hat{u}_k)_y^\xi,I \cap (B-\varepsilon_k\xi)_y^\xi) +
     \liminf_{k \to  \infty } F_{\varepsilon_k}((\hat{u}_k)_y^\xi,B_y^\xi \setminus I \cap (B-\varepsilon_k\xi)_y^\xi) 
     \\
     &
     \ge \liminf_{k \to  \infty } F_{\varepsilon_k}((\hat{u}_k)_{y,\lambda}^\xi,I \cap (B-\varepsilon_k\xi)_y^\xi) +
     \liminf_{k \to  \infty } F_{\varepsilon_k}((\hat{u}_k)_{y,\lambda}^\xi,B_y^\xi \setminus I \cap (B-\varepsilon_k\xi)_y^\xi) 
     \\
     & \ge \frac{\pi}{2}\mathcal{H}^0(\partial^*A_y^\xi \cap I)+\frac{\pi}{2} \mathcal{MS}_{\frac{2}{\pi}}(\hat{u}^{\xi}_{y,\lambda},B_y^\xi \setminus I)
     \\
     & =\frac{\pi}{2}\mathcal{H}^0(\partial^*A_y^\xi \cap I) + \frac{\pi}{2}\mathcal{H}^0(J_{\hat{u}^\xi_{y,\lambda}}\setminus I) + \int_{B_y^\xi \setminus I}|\nabla \hat{u}^{\xi}_{y,\lambda}(t)|^2 \,\de t.
\end{align*}
Since the above inequality holds for every $\delta >0$ and $t \in \partial^* A_y^\xi$, we get
\begin{align*}
     \liminf_{k \to  \infty} F_{\varepsilon_k}((\hat{u}_k)_y^\xi,B_y^\xi \cap (B-\varepsilon_k\xi)_y^\xi)\ge \frac{\pi}{2} \mathcal{H}^0(\partial^*A_y^\xi \cup J_{\hat u_{y,\lambda}^\xi})+\int_{B_y^\xi }|\nabla \hat{u}^{\xi}_{y,\lambda}(t)|^2 \,\de t.
\end{align*}
Since $\tau_\lambda(t)=t$ for $t \in [-\lambda/2,\lambda/2]$ and $J_{\hat{u}^\xi_{y,\lambda}} =J_{\hat{u}^\xi_{y}}$, we infer that
\begin{align*}
     \liminf_{k \to  \infty } &F_{\varepsilon_k}((\hat{u}_k)_y^\xi,B_y^\xi \cap (B-\varepsilon_k\xi)_y^\xi) \ge \frac{\pi}{2} \mathcal{H}^0(\partial^*A_y^\xi \cup J_{\hat u_{y,\lambda}^\xi})+\int_{B_y^\xi }|\nabla \hat{u}^{\xi}_{y,\lambda}(t)|^2 \,\de t\\
     &\ge\frac{\pi}{2} \mathcal{H}^0(\partial^*A_y^\xi \cup J_{\hat u_{y}^\xi}) + \int_{\big\{ t \in B_y^\xi:| \hat{u}^{\xi}_{y,\lambda}(t)|\le \lambda/2\big\}}|\nabla \hat{u}^{\xi}_{y}(t)|^2 \,\de t .
\end{align*}
Now, by sending $\lambda \to \infty$ we get for every $c \in \R$
\begin{equation*}
\begin{split}
     \liminf_{k \to  \infty} F_{\varepsilon_k}((\hat{u}_k)_y^\xi,B_y^\xi \cap (B-\varepsilon_k\xi)_y^\xi) &\ge \frac{\pi}{2} \mathcal{H}^0(\partial^*A_y^\xi \cup J_{\hat u_{y}^\xi}) +\int_{B_y^\xi}|\nabla \hat{u}^{\xi}_{y}(t)|^2 \,\de t\\
     &\ge\frac{\pi}{2}\mathcal{MS}_{\frac{2}{\pi}}(\hat u_y^\xi (1-\chi_{A_y^\xi})+c\chi_{A_y^\xi},B_y^\xi).
\end{split}
\end{equation*}
By combining the above inequality with Fatou's Lemma we deduce
\begin{align*}
&\sup_{\mathscr{B}} \sum_{B \in \mathscr{B}}\frac{\pi}{2}\left (\int_{\R^n} \left (\int_{\Pi^\xi} \mathcal{MS}_{\frac{2}{\pi}}(\hat u_y^\xi (1-\chi_{A_y^\xi})+c \chi_{A_y^\xi},B_y^\xi) \de \mathcal{H}^{n-1}(y) \right )^{p}|\xi|^p e^{-|\xi|^2} \de \xi\right )^{\frac{1}{p}}\\
&\le \liminf_{k \to \infty}\sup_{\mathscr{B}} \sum_{B \in \mathscr{B}}\left (\int_{\frac{\Omega-\Omega}{\varepsilon_k}} \left (\int_{\Pi^\xi} F_{\varepsilon_k}((\hat{u}_k)_y^\xi,B_y^\xi \cap (B-\varepsilon_k\xi)_y^\xi)\de \mathcal{H}^{n-1}(y) \right )^{p}|\xi|^p e^{-|\xi|^2} \de \xi\right )^{\frac{1}{p}}\\
&=\liminf_{k \to \infty} \mathcal{F}^p_{\varepsilon_k}(u_k,\Omega) <\infty,   
\end{align*}
which implies that $\hat{\mu}_u^p(\Omega)<\infty$, and hence, by Remark \ref{r:mup>mu1}, also that
$u \in \text{GSBV}^{\mathcal{E}}(\Omega;\R^n)$. Thus by Theorem \ref{t:main-emanuele} we conclude that $u \in \text{GSBD}(\Omega)$.
In addition, recall that, by Lemma~\ref{limite}, we have for every $c\in \R$
    \begin{align*}
         \sup_{\mathscr{B}} \sum_{B \in \mathscr{B}}\frac{\pi}{2}&\left (\int_{\R^n} \left (\int_{\Pi^\xi} \mathcal{MS}_{\frac{2}{\pi}}(\hat u_y^\xi (1-\chi_{A_y^\xi})+c \chi_{A_y^\xi},B_y^\xi) \de \mathcal{H}^{n-1}(y) \right )^{p}|\xi|^p e^{-|\xi|^2} \de \xi\right )^{\frac{1}{p}}\\
         &=\!\int_{\Omega}  \varphi_{p}  (e(u_c(1-\chi_A))) \de x +  \beta_{p}  \mathcal{H}^{n-1}(J_{u_c}),
    \end{align*} 
where $u_c\colon\Omega \to \R^n$ is defined as $u_c:=u(1-\chi_A)+c \chi_A$. Since a measure theoretic argument yields that $\partial^*A \subset J_{u_c}$ for a.e. $c \in \R$, inequality \eqref{semi378} follows.

\section{$\Gamma$-convergence and convergence of quasi-minimisers}
\label{s:gamma-conv}

This section is devoted to the proof of Theorem~\ref{t:maingamma} as well as to the convergence, up to subsequences, of quasi-minimizers of~$\mathcal{F}^{p}_{\varepsilon}$ under Dirichlet boundary conditions. The former is performed in Section~\ref{sub:Gamma}, the latter in Section~\ref{sub:minimizers}.

\subsection{$\Gamma$-convergence}
\label{sub:Gamma}

We start with the following pointwise convergence result following the strategy in~\cite[Theorem~3.4]{Gobbino}.
\begin{proposition}
    Let $I \subset \R$ be an interval  and let  $u \in \textnormal{GSBV}(I)$. For every $\varepsilon>0$, it holds
    \begin{equation}\label{3.22gob}
    F_{\varepsilon}(u,I \cap (I-\varepsilon ))\le \frac{\pi}{2}\mathcal{MS}_{\frac{2}{\pi}}(u,I).
    \end{equation}
\end{proposition}

\begin{proof}
Since $I$ is an interval, the set $I \cap (I - \varepsilon)$ is also an interval or it is empty. Assume that it is not empty otherwise there is nothing to prove.
Define $A^\varepsilon:= \{t \in I \cap (I -\varepsilon) : [t,t+\varepsilon] \cap J_u \neq \emptyset \}$. Notice that $t \in I \cap (I -\varepsilon)$ implies $[t,t+\varepsilon] \subset I$. We have
\[
\frac{1}{\varepsilon}\int_{A_\varepsilon} \arctan \bigg (\frac{(u(t+\varepsilon)- u(t))^2}{\varepsilon} \bigg ) \de t \le \frac{\pi}{2 \varepsilon} |A_\varepsilon| \le \frac{\pi}{2} \mathcal{H}^0(J_u)\,.
\]
 Since $\arctan x \le x$ for $x \ge 0$,  we have
\begin{align*}
&\frac{1}{\varepsilon}\int_{(I \cap (I-\varepsilon))\setminus A_\varepsilon} \arctan \bigg (\frac{(u(t+\varepsilon)- u(t))^2}{\varepsilon} \bigg ) \de t \\
&\le \frac{1}{\varepsilon} \int_{(I \cap (I-\varepsilon))\setminus A_\varepsilon} \arctan \bigg (\int_t^{t + \varepsilon} |\nabla u(\tau)|^2\, \de\tau  \bigg ) \de t\\
&\le \frac{1}{\varepsilon} \int_{(I \cap (I-\varepsilon))\setminus A_\varepsilon} \int_t^{t + \varepsilon} |\nabla u(\tau)|^2\, \de\tau \,\de t\le \int_{I} |\nabla u(t)|^2 \de t.
\end{align*}
By combining the above two inequalities we conclude the proof.
\end{proof}

We are now in a position to conclude the proof of Theorem~\ref{t:maingamma}.

\begin{proof}[Proof of Theorem~\ref{t:maingamma}]
    We notice that the $\Gamma$-liminf inequality is a direct consequence of Theorem~\ref{compattezza} (cf.~\eqref{semi378}). To conclude for the $\Gamma$-convergence, it is enough to show that for every $u \in L^{0} (\Om; \R^{n})$ it holds true
    \begin{equation}
    \label{gamma2}
        \lim_{\varepsilon \to 0} \mathcal{F}^{p}_{\varepsilon} (u, \Om) \leq \mathcal{F}^{p} (u, \Om)\,.
    \end{equation}
    We begin by recalling that
    \begin{equation*}
    \begin{split}
    \mathcal{F}^p_{\varepsilon}(u,\Omega)&=\sup_{\mathscr{B}}\sum_{B \in \mathscr{B}}\left (\int_{\frac{\Omega -\Omega}{\varepsilon}} F_{\varepsilon,\xi}(u,B)^p e^{-|\xi|^2} \de \xi\right )^{\frac{1}{p}}\\
    &=\sup_{\mathscr{B}}\sum_{B \in \mathscr{B}}\left (\int_{\frac{\Omega -\Omega}{\varepsilon}} \left (\int_{\Pi^\xi}F_{\varepsilon}(\hat u_y^\xi,B_y^\xi \cap (B-\varepsilon\xi)_y^\xi) \de y \right )^p |\xi|^p e^{-|\xi|^2} \de \xi\right )^{\frac{1}{p}},
    \end{split}
    \end{equation*}
    where the supremum is taken over  the set of finite families of disjoint open balls contained in $\Omega$.
     Then, since $(B-\varepsilon\xi)^\xi_y=B^\xi_y -\varepsilon$ and since $B^\xi_y \subset \mathbb{R}$ is an interval, we combine the above equation with \eqref{3.22gob} and we deduce 
    \begin{align*}
    \limsup_{\varepsilon\to 0}\mathcal{F}^p_{\varepsilon}(u,\Omega) &\le \limsup_{\varepsilon\to 0} \sup_{\mathscr{B}}\sum_{B \in \mathscr{B}} \frac{\pi}{2} \left (\int_{\frac{\Omega -\Omega}{\varepsilon}} \left (\int_{\Pi^\xi} \mathcal{MS}_{\frac{2}{\pi}}(\hat u_y^\xi,B_y^\xi ) \de y \right )^p |\xi|^p e^{-|\xi|^2} \de \xi \right )^{\frac{1}{p}}\\
    &\le  \sup_{\mathscr{B}}   \sum_{B \in \mathscr{B}}  \frac{\pi}{2} \left (\int_{\R^n} \left (\int_{\Pi^\xi} \mathcal{MS}_{\frac{2}{\pi}}(\hat u_y^\xi, B_y^\xi  ) \de y \right )^p |\xi|^p e^{-|\xi|^2} \de \xi\right )^{\frac{1}{p}}.
    \end{align*}
    Thanks to Lemma~\ref{limite} we infer~\eqref{gamma2}. This concludes the proof of the theorem.
\end{proof}

\subsection{Convergence of quasi-minimisers}
\label{sub:minimizers}

For the purposes of this subsection, we fix two open sets $\Omega \subset \Omega' \subset \mathbb{R}^n$. 
We further assume that $\partial_D \Omega$, namely, the Dirichlet part of the boundary, satisfies $\partial_D \Omega= \partial \Omega \cap \Omega'$. As it is  customary in free discontinuity problems,   
we consider a relaxed boundary condition on $\partial_D \Omega$. To this purpose, our Dirichlet datum is given by a function $f \colon \partial_D \Omega \to  \R^n $ which we identify, with a slight abuse of notation, with the trace on~$\partial_{D} \Omega$ of a function $f \in H^{1}(\Omega'; \R^n)$. 
The domain of our functionals is thus denoted by $L^0_f(\Omega';\R^n)$ and defined as
\[
L^0_f(\Omega';\R^m):= \big \{u \in L^0(\Omega';\R^n) : u=  f  \text{ a.e. in }\Omega'\setminus \Omega\big \}.
\]
In addition we define for every $\varepsilon >0$ the functionals $\mathcal{F}^f_{\varepsilon}\colon L^{0} (\Om'; \R^{n}) \to [0, \infty]$ as
\[
\mathcal{F}^{p,f}_{\varepsilon}(u,\Omega') :=
\begin{cases}
   \mathcal{F}^p_{\varepsilon}(u,\Omega') & \text{if }  u \in L^0_f(\Omega';\mathbb{R}^n) \\[1mm]
  + \infty & \text{otherwise in } L^0(\Omega';\mathbb{R}^n) 
\end{cases}
\]
and the limit functional $\mathcal{F}^{p,f} \colon L^{0} (\Om'; \R^{n}) \to [0, \infty] $ as
\[
\mathcal{F}^{p,f}(u,\Omega') :=
\begin{cases}
   \mathcal{F}^p(u,\Omega') & \text{if }  u \in  \text{GSBD}(\Omega')  \cap L^0_f(\Omega';\mathbb{R}^n) \\[1mm]
  + \infty & \text{otherwise in } L^0(\Omega';\mathbb{R}^n).
\end{cases}
\]
We further observe that the functional $\mathcal{F}^{p, f}$  takes the form 
\[
   \mathcal{F}^{p, f}  (u,\Omega')= \int_{\Omega'}  \varphi_{p}  (e (u))   \de x +  \beta_{p}  \bigg(\mathcal{H}^{n-1}(J_u \cap \Omega) + \mathcal{H}^{n-1}(J_u \cap \partial_D \Omega)\bigg),
\]
and that the term $\mathcal{H}^{n-1}(J_u \cap \partial_D \Omega)$ penalizes the part of $\partial_D \Omega$ on which the Dirichlet condition is not attained, namely, the set $\{x \in \partial_D \Omega  : u^-(x) \neq f(x) \}$ where $u^-(x)$ denotes the trace with respect to the inner normal to $\partial \Omega$ of $u$ at $x$ .

We have the following theorem.

\begin{theorem}
    \label{t:convqmin}
     Let   $\Omega, \Om' \subset \R^n$ be two open sets, let  $\partial_D \Omega = \partial \Omega \cap \Om'$  be the Dirichlet part of the boundary,  and let $f \in H^{1}(\Omega ' ; \R^n)$  be an admissible Dirichlet datum as above. Assume that $\{u_\varepsilon\}_{\varepsilon >0}$ is  a family of quasi-minimisers for $\{\mathcal{F}^{p,f}_\varepsilon\}_{\varepsilon>0}$, namely,
    \begin{equation}
    \label{e:quasimin}
\lim_{\varepsilon \to 0} \bigg( \mathcal{F}_\varepsilon^{p,f}(u_\varepsilon,\Omega')- \inf_{w \in L^0(\Omega';\R^n)} \mathcal{F}_\varepsilon^{p,f}(w,\Omega') \bigg) = 0\,.  
    \end{equation}
    Then, there exist a minimizer $u \in {\rm GSBD} (\Om') \cap L^{0}_{f} (\Om'; \R^{n})$ of~$\mathcal{F}^{p, f}$ and a subsequence $\varepsilon_{k} \to 0$ as $k \to \infty$ such that $u_{\varepsilon_k}\to u$ almost everywhere in $\Omega'$. 
\end{theorem}

\begin{proof}
By \eqref{gamma2} we have that $\sup_{\varepsilon >0}\mathcal{F}^{p,f}_\varepsilon(f,\Omega') < \infty$.  Hence, condition \eqref{e:quasimin} yields $\sup_{\varepsilon >0}\mathcal{F}^{p,f}_\varepsilon(u_\varepsilon,\Omega') < \infty$. We are thus in position to apply Theorem~\ref{compattezza} to infer the existence of a subsequence $\varepsilon_k \to 0$ as $k \to \infty$ and a  finite  perimeter set $A \subset \Omega'$  such that $A = \{ x \in \Om': \, |u_{\varepsilon_{k}}| \to \infty \text{ as $k \to \infty$} \}$, $u_{\varepsilon_k} \to u$ pointwise a.e. in $\Omega' \setminus A$ for $u \in  \text{GSBD}  (\Omega')$ with $u=0$  in~$A$.  In addition,~\eqref{semi378} is  fulfilled. Since $u_\varepsilon = f$ a.e.~in~$\Omega' \setminus \Omega$ for every 
$\varepsilon>0$, we deduce $A \subset \Omega$. 

It remains to prove that $u$ is actually a minimiser of $\mathcal{F}^{p,f}(\cdot,\Omega')$. We observe that we cannot immediately infer that $u$ is a minimiser from the $\Gamma$-convergence given by Theorem~\ref{t:maingamma}, since the sequence  $u_{\varepsilon_k}$ does not convergence in measure to $u$ in the whole of~$\Omega'$. Nevertheless,  the proof can be achieved with the very same argument as in the standard case. Indeed, take  $w \in  L^{0} (\Om'; \R^{n})$.  Since we want to prove that $\mathcal{F}^{p,f}(w,\Omega') \geq \mathcal{F}^{p,f}(u,\Omega')$, we assume with no loss of generality that  $w \in  \text{GSBD}(\Omega') \cap L^{0}_{f} (\Om'; \R^{n})$.  By virtue of \eqref{semi378}, \eqref{gamma2}, and \eqref{e:quasimin}, we  estimate
\begin{align*}
    \mathcal{F}^{p,f}(w,\Omega') &=\lim_{k \to \infty} \mathcal{F}^{p,f}_{\varepsilon_k}(w,\Omega') \geq \limsup_{k \to \infty} \inf_{ v \in L^0(\Omega';\R^m)} \mathcal{F}_\varepsilon^{p,f}( v,\Omega') \\
    & = \limsup_{k \to \infty}\mathcal{F}_\varepsilon^{p,f}(u_{\varepsilon_k},\Omega') \geq \int_{\Omega'}  \varphi_{p} (e(u))  \, \de x +  \beta_{p}  \mathcal{H}^{n-1}(J_u \cup \partial^* A) \\
     &\geq \int_{\Omega'}  \varphi_{p} (e(u))  \, \de x  +  \beta_{p}   \mathcal{H}^{n-1}(J_u)\\
     &=
     \int_{\Omega'}  \varphi_{p} ( e(u))  \,  \de x +  \beta_{p}  \bigg(\mathcal{H}^{n-1}(J_u \cap \Omega) + \mathcal{H}^{n-1}(J_u \cap \partial_D \Omega)\bigg) = \mathcal{F}^{p,f}(u,\Omega').
\end{align*}
Thanks to the arbitrariness of $w$  we deduce that $\partial^{*}A \subset J_{u}$ and that $u$ is a minimizer of~$\mathcal{F}^{p, f}$.     
\end{proof}

\appendix
\section{}\label{appenda}

\begin{proof}[Proof of Proposition \ref{p:complete}]
    For every measurable set $B \subset A$ we can consider the linear map $i_B \colon L^0(A) \to L^0(B)$ defined as $i(v):= v \restr B$. Clearly, being $V \subset L^0(A)$, the map $i_B$ is a well defined linear map between the vector spaces $V$ and $V(B)$, where we have set $V(B):= \{u \in L^0(B) : u=v \restr B, \text{ for some }v \in V \}$. We claim that there exists $\varepsilon_c >0$ such that for every measurable set $K \subset A$ with $\mathcal{L}^n(A \setminus K) \leq \varepsilon_c$ we have that the linear map $i_K \colon V \to V(K)$ is injective. Indeed, fix a basis $\{v_1,\dotsc,v_k\}$ for $V$ and assume by contradiction that there exists $\varepsilon_j \searrow 0$ and measurable sets $K_j \subset A$ with $\mathcal{L}^n(A \setminus K_j) \leq \varepsilon_j$ such that there exists real coefficients, not all identically zero, $\{\alpha^j_1,\dotsc,\alpha^j_k\}$, satisfying 
    \begin{equation}
    \label{e:lindep}
    \sum_{i=1}^k \alpha^j_i (v_i \restr K_j) =0. 
    \end{equation}
    By renormalization, with no loss of generality, we may assume that for every $i=1,\dotsc,k$ and every $j=1,2,\dotsc$ it holds true $|\alpha^j_i| \leq 1$ and there exists at least one $i(j)=1,\dotsc,k$ such that $|\alpha^j_{i(j)}|=1$. Up to pass to a not relabelled subsequence in $j$, we may suppose that the index $i(j)$ does not depend on $j$, and that $\alpha^j_i \to \alpha_i$ for every $i=1,\dotsc,n$, for some real coefficients $\{\alpha_1,\dotsc,\alpha_k\}$ which are not all identically zero. Therefore, since $K_j \nearrow A$ in measure we deduce that $\sum_{i=1}^k \alpha^j_i i_{K_j}(v_i)$ (extended to zero out of $K_j$) converges to $\sum_{i=1}^k \alpha_i v_i$ pointwise a.e. in $A$ as $j \to \infty$. From \eqref{e:lindep} we immediately deduce that $\sum_{i=1}^k \alpha_i v_i=0$ in $L^0(A)$ which gives a contradiction. The claim is thus proved.

    Now assume that $\{w_j\}_{j \in \N} \subset V$ is a Cauchy sequence. Since $L^0(A)$ is complete we have ${\rm d}(w_j,w) \to 0$ as $j \to \infty$ for some $w \in L^0(A)$. We want to prove that $w \in V$. Up to passing to a subsequence in $j$, the ${\rm d}$-convergence implies $w_j \to w$ pointwise a.e. in $A$ as $j \to \infty$. Therefore, for every $0<\varepsilon \leq \varepsilon_c$, we apply Egorov's Theorem to find a measurable set $K \subset A$ such that $\mathcal{L}^n(A \setminus K) \leq \varepsilon$ and $w_j \to w$ uniformly on $K$ as $j \to \infty$. But this means that the sequence $\{i_K(w_j)\}_{j\in \N}$ is a sequence belonging to the finite-dimensional vector space $V(K)$ and converging with respect to the $L^1$-norm. Hence, we find real coefficients $\{\alpha_1^K, \dotsc,\alpha_k^K\}$ such that $i_K(w)= \sum_{i=1}^k \alpha_i^K i_K(v_i)$, where $\{v_1,\dotsc, v_k\}$ is a basis of $V$. In order to conclude we need to show that, given a sequence of measurable sets $K_j \nearrow A$ the coefficients $\{\alpha_1^{K_j}, \dotsc,\alpha_k^{K_j}\}$ are definitely constant as $j \to \infty$. So assume by contradiction that for every $j\geq 1$ we find $j < j_1 < j_2$ such that $\{\alpha^{K_{j_1}}_1, \dotsc, \alpha^{K_{j_1}}_k\} \neq \{\alpha^{K_{j_2}}_1, \dotsc, \alpha^{K_{j_2}}_k\}$. Now choose $j_1$ so large such that $\mathcal{L}^n(A \setminus (K_{j_1} \cap K_{j_2})) \leq \varepsilon_c$. By our assumption on the non-coincidence between the $K_{j_1}$-coefficients and the $K_{j_2}$-coefficients, we deduce that the linear combination $\sum_{i=1}^k (\alpha^{K_{j_1}}_i -\alpha^{K_{j_2}}_i) i_{K_{j_1} \cap K_{j_2}}(v_i)$ is identically null in $L^0(K_{j_1} \cap K_{j_2})$ while its coefficients are not all identically equal to zero. But this gives a contradiction with the fact that, thanks to our previous claim, the linear map $i_{K_{j_1} \cap K_{j_2}} \colon V \to V(K_{j_1} \cap K_{j_2}) \subset L^0(K_{j_1} \cap K_{j_2})$ is injective. 
\end{proof}

\section*{Acknowledgements}
 We would like to thank Antonin Chambolle and Vito Crismale for pointing out to us, after the first version of this paper was written, that the identification of the spaces $GSBD$ and $GSBV^{\mathcal{E}}$ was indeed possible. This observation encouraged us to develop an alternative approach. 

This research has been supported by the Austrian Science Fund (FWF) through grants  \href{https://doi.org/10.55776/F65}{10.55776/F65},  \href{https://doi.org/10.55776/Y1292}{10.55776/Y1292}, \href{https://doi.org/10.55776/P35359}{10.55776/P35359}, \href{https://doi.org/10.55776/F100800}{10.55776/F100800}, by the OeAD-WTZ project CZ04/2019 (M\v{S}MT\v{C}R 8J19AT013), by the Italian Ministry of Research through the PRIN 2022 project No.~2022HKBF5C ``Variational Analysis of Complex Systems in Materials Science, Physics and Biology'', and by the University of Naples Federico II through the FRA Project ``ReSinApas''. S.A. and A.K. are also members of the Gruppo Nazionale per l'Analisi Matematica, la Probabilit\`a e le loro Applicazioni (INdAM-GNAMPA). For open access purposes, the authors have applied a CCBY public copyright license to any accepted manuscript version arising from this submission.

\subsection*{Conflict of interest:} The Authors declare no conflict of interest.
 
\subsection*{Data availability:} The manuscript contains no associated data.

\printbibliography

\end{document}